\documentclass[sn-mathphys-num]{sn-jnl}
\usepackage[dvipsnames]{xcolor}

\usepackage{mathtools}
\usepackage{booktabs}
\usepackage{pgfplots}
\usepgfplotslibrary{groupplots}
\usepgfplotslibrary{fillbetween}
\pgfplotsset{
	every mark/.append style={solid},
}
\usepackage{siunitx}
\usepackage{amsmath}
\usepackage{amsthm}
\usepackage{amsfonts}
\usepackage{dsfont}
\usepackage{multicol}
\usepackage{multirow}

\usepackage{caption}
\usepackage{subcaption}

\usepackage{listings}

\lstdefinelanguage{Julia}%
{morekeywords={abstract,break,case,catch,const,continue,do,else,elseif,%
		end,export,false,for,function,immutable,import,importall,if,in,%
		macro,module,otherwise,quote,return,switch,true,try,type,typealias,%
		using,while},%
	sensitive=true,%
	alsoother={\$},%
	morecomment=[l]\#,%
	morecomment=[n]{\#=}{=\#},%
	morestring=[s]{"}{"},%
	morestring=[m]{'}{'},%
}[keywords,comments,strings]%

\definecolor{codegreen}{rgb}{0,0.6,0}
\definecolor{codegray}{rgb}{0.5,0.5,0.5}
\definecolor{codepurple}{rgb}{0.58,0,0.82}
\definecolor{backcolour}{rgb}{0.95,0.95,0.92}

\lstdefinestyle{mystyle}{
    backgroundcolor=\color{backcolour},   
    commentstyle=\color{codegreen},
    keywordstyle=\color{magenta},
    numberstyle=\tiny\color{codegray},
    stringstyle=\color{codepurple},
    basicstyle=\ttfamily\footnotesize,
    breakatwhitespace=false,
    breaklines=true,
    captionpos=b,
    keepspaces=true,
    numbers=left,
    numbersep=5pt,
    showspaces=false,
    showstringspaces=false,
    showtabs=false,
    tabsize=2,
    lineskip=2pt    
}

\title[Stagnation in floating point]{Analysis and Detection of Stagnation in Floating-Point Summation}

\DeclareMathOperator{\fl}{\mathrm{fl}}
\DeclareMathOperator{\ulp}{\mathrm{ulp}}
\DeclareMathOperator{\R}{\mathbb{R}}
\DeclareMathOperator{\F}{\mathbb{F}}

\theoremstyle{definition}
\newtheorem{definition}{Definition}
\theoremstyle{theorem}
\newtheorem{theorem}{Theorem}
\newtheorem{proposition}[theorem]{Proposition}

 
\begin{document}

\author[1]{\fnm{Mantas} \sur{Mikaitis}}
\author*[1]{\fnm{William} \sur{Woolfenden}}\email{rcvw5971@leeds.ac.uk}

\affil[1]{\orgdiv{School of Computer Science}, \orgname{University of Leeds}, \orgaddress{\city{Leeds}, \country{United Kingdom}}}

\abstract{
  \emph{Stagnation} in the summation of floating-point numbers occurs when the running sum, or the accumulator, does not change its value despite addends being greater than zero.
  Here we define different types of stagnation---commonly referred to as \emph{swamping} or \emph{absorption}, where one operand in an addition is nearly or completely rounded off---in floating-point summation, and propose algorithms for detecting it in numerical code that uses floating-point arithmetic.
  Whilst \emph{cancellation} in floating-point arithmetic has been studied, for example by Lam et al.~\cite{lhs13}, stagnation has not been exploited as a property for detecting and reporting accuracy loss in floating-point operations.
  We have built prototype dynamic stagnation detection tools in two languages, MATLAB and Julia, and applied them to applications ranging from basic summation of a vector, to matrix multiplication, to a fluid flow simulation model.
  With this tool, users are able to identify areas of their floating-point code where stagnation occurs---in those areas of the code one may need to switch to a higher precision number format, apply stochastic rounding, or use different summation algorithms such as compensated summation.
  To the best of our knowledge we provide a first rigorous definition of stagnation in floating-point summation.
}

\maketitle


\section{Introduction}

Take $\F_{<e_{min}, e_{max}, p>} \subset \R$ to be an IEEE 754~\cite{ieee19} compliant binary floating-point subset of the reals, where $e_{min}$ is the minimum exponent, $e_{max}$ is the maximum exponent, and $p$ is the precision.
Each value in this set is represented as $\pm 2^em$ where $e_{min} \leq e \leq e_{max}$ and $0 \leq m \leq 1.1\dots1_2$ is represented using $p$ binary digits.
We will use $\F$ as a shorthand for  $\F_{<e_{min}, e_{max}, p>}$, unless we need several different floating-point formats in our discussion.

Given $a, b \in \F$ such that $|a| \geq |b|$, a homogeneous floating-point addition operation is defined as $x:=\fl(a+b)$, where $x$ is also a precision-$p$ floating-point value and $\fl()$ denotes rounding-to-nearest (ties to even)~\cite{ieee19} of the sum $a+b \in \R$ to $x \in \F$.
Each of $x$, $a$, and $b$ have $p$ bits in their significands, including the implicit bit, following the IEEE 754 standard~\cite{ieee19}.

Take $e_a$ and $e_b$ to be the exponents of $a$ and $b$.
One of the steps performed by floating-point adders is to compute the difference in the exponents, $d=e_a-e_b$ and shift the significand of $b$ by $d$ places to the right before performing the addition of it to the significand of $a$.
The top $u=\max(0, p-d)$ bits of the significand of $b$, padded with zeros, are used in the $p$-bit integer addition.
Addition is followed by an optional normalization step where one bit of the result may be shifted out.
Overall, in the significand alignment and normalization, up to $d+1$
least significant bits of the significand of $b$ are not utilised in forming the sum---their value is approximated by rounding the GRT (Guard, Round, Sticky) bits~\cite[Sec.~8.4.3]{erla04}.

We now define \emph{partial stagnation}, \emph{total stagnation}, and \emph{near total stagnation}, given any $n$ floating-point numbers $x_1, x_2, \dots, x_n \in \F$ added by $n-1$ addition operations using \emph{recursive summation} algorithm~\cite[Sec.~4.1]{high:ASNA2}: $\widehat{s_{i}} = \fl(\widehat{s_{i-1}}+x_{i})$ for $i = 2, \mathellipsis, n$ with $\widehat{s_1}=s_1=x_1$.
In general, $\widehat{s_i}$ is the floating point approximation of the actual sum $s_i$.
Here we call $\widehat{s_{i-1}}$ the \emph{accumulator}.
Each step $i$ has a corresponding $u_i$ which depends on the exponents and magnitudes of $\widehat{s_{i-1}}$ and $x_i$, and precision $p$.

\begin{definition}[Partial stagnation]
  Partial stagnation occurs over a set of $n-1$ addition operations, when $0<u_i\leq t$ for some $0 < M \leq n-1$ additions and for some given threshold $0 < t \leq p$. Here $t$ is the maximum number of most significant bits used in each of the $M$ addition operations for which the above conditions hold.
\end{definition}

\begin{definition}[Total stagnation]
  \label{def:tot-stagn}
  Total stagnation occurs over a set of $n-1$ addition operations, when for any $0 < M \leq n-1$ operations $u_i = 0$ and $\fl(\widehat{s_{i-1}} + x_{i}) \in \{\widehat{s_{i-1}}, x_{i}\}$.
\end{definition}

\begin{definition}[Near-total stagnation]
  Near-total stagnation occurs over a set of $n-1$ addition operations, when for any $0 < M \leq n-1$ addition operations $\fl(\widehat{s_{i-1}} + x_i) \notin \{\widehat{s_{i-1}}, x_{i}\}$, but $u_i=0$ (the addend with the larger magnitude is changed only by rounding).
\end{definition}

It is useful to differentiate these cases: \emph{total stagnation} tells us when the additions have no effect as the accumulator is set to one of the two addends; \emph{near-total stagnation} is a good indication that the changes being made contain very little useful information as all of them are a result of rounding; lastly, \emph{partial stagnation} is a good indication of how much information in the addends of smaller magnitudes gets used and is a parameterized measure through $t$.
In some circumstances, partial stagnation may transition into total stagnation if the magnitude of the accumulator is growing, such as when approximating the harmonic series~\cite{malo13}.
The conditions for near-total stagnation include that of total stagnation, and the conditions for partial stagnation include that of both near-total and total stagnation.

Note that $M$ additions, out of the $n-1$, need not be consecutive.
The conventional understanding of ``stagnation'', in informal statements, is applied to a sequence of consecutive addition operations.
We instead consider that in a sequence of $n-1$ addition operations, the $M$ additions that apply to the definition of stagnation may appear anywhere in the summation.
Alternative approaches are discussed briefly later in \ref{sssec:consecutive}.

Total stagnation is the most important type of the three defined, since it refers to complete loss of information from addends.
Partial and near-total stagnation can still apply to a summation where the accumulator changes by small amounts, but for total stagnation the accumulator never grows since operands are rounded off completely -- it is the most severe kind.
In our work we focus on total stagnation.
Future work could explore the effectiveness of measuring partial and near-total stagnation.

It is worth noting that when a single instance of an operand is rounded off, this is often called \textit{absorption}.
Using this terminology, total stagnation of Definition~\ref{def:tot-stagn} is comprised of $M$ absorptions.
This is analogous for other types of stagnation, when the term absorption is used to refer to most of the smaller addend's significand being lost to rounding.

We define the \textit{unit in the last place} of a floating point number, often expressed as $\ulp$.
\begin{definition}[Unit in the last place]
  \label{def:ulp}
  Let $x \in \mathds{F}$ with precision $p$, where $x = \pm 2^{e}m$. The unit in the last place is defined as $\ulp(x) = 2^{-(p-1)} \times 2^e$.
\end{definition}
When performing addition $x+y$ where $|x| > |y|$ in floating point with round-to-nearest, total absorption occurs if $|y| < |\ulp(x)|/2$.

We denote the average $(\sum_{i=1}^{n-1}u_i)/(n-1)$ as the
\emph{average level of stagnation}, taking $t=p$.
Taking smaller $t$ requires taking a subset of all additions $n-1$ that partially or totally stagnated.
This can be a good measure of information loss in a summation: a smaller value means more bits shifted out in additions that partially or totally stagnated, while a larger value means more bits were effectively used in additions (average $u$ is close to $t$).
This measure can tell us useful information about the numerical code without using a reference solution which is required in standard error comparison benchmarks.
Gathering statistics through $u$ about individual operations can be useful in detecting which parts of a numerical code suffer from loss of bits the most and indicate that those parts require the use of extra precision to increase $u$, a different summation algorithm, or require applying stochastic rounding~\cite{cfhm21}.

\section{Previous work}

Benz, Hildebrandt, and Hack~\cite{bhh12} and Lam, Hollingworth, and Stewart~\cite{lhs13} developed floating-point cancellation detection~\cite[Sec.~1.7]{high:ASNA2} in the addition operations at run time in order to report various statistics, such as the number of occurences and how many bits have been cancelled.
In $\fl(a+b)$ cancellation is detected through inspecting the exponents of arguments ($e_a$, $e_b$) and the result ($e_r$): it occurs when $e_r < \mathrm{max}(e_a, e_b)$.
Their measure of cancellation badness of a subtraction operation measures how much more precision would be needed for the cancellation to not be catastrophic.
From \cite{lhs13}: ``\emph{It is important to emphasize that cancellation is usually not a killer but instead a death certificate that reveals an earlier loss of information}''.
Using this analogy, stagnation is a killer---information exists in the addends before stagnation erases it over many additions.

The FPDebug tool provided by Benz, Hildebrandt, and Hack~\cite{bhh12} measures the cancellation severity for every variable and floating-point addition or subtraction instruction in a program.
A side-by-side computation is run in high precision, storing information on operations and high precision variable counterparts in a data structure called a \textit{shadow value}.
The shadow value can be used to measure how many bits in a variable are exact, and combining this with measuring the bits cancelled tells us the cancellation severity.
This can be tracked for every instance of cancellation by FPDebug providing pointers to where the cancelling operations take place.

Cancellation detection is implemented by Lam, Hollingworth, and Stewart~\cite{lhs13} similarly, measuring the cancellation severity of each instruction in a program.
In every addition and subtraction operation, the exponents of the operands and the result are recorded, and a \textit{priority} $\max(e_1, e_2) - e_{out}$ is measured. 
The priority will be positive if cancellation has occurred.
If the priority is equal to or above a threshold, which is $10$ by default, the cancellation is reported to a file.
The example programs use double precision.

Lopes et al.~\cite{lfdc18, lfhd18} have developed static code analysis tools to report which parts of the code are subject to cancellation and \emph{absorption}.
They focus on distinct addition or subtraction operations and define absorption on ranges $[min_1, max_1]$ and $[min_2, max_2]$ of the two floating-point values that are added or subtracted.
Assuming the first range constitutes the range of the addend smaller in magnitude, the authors focused on binary32 and defined absorption to occur when $\min(min_1, max_1)/\max(min_2, max_2) < 2^{-24}$.
This is a definition that is local to a single addition operation, and is fundamentally different from Definition~\ref{def:tot-stagn} of stagnation over a series of additions.
The software tool of Lopes et al. is a static code analyser, which requires users to specify the ranges of the input values that are then propagated through the arithmetic operations to determine if stagnation can occur given the ranges of inputs of each addition operation.
The benefits of this approach are that running the program is not required, but the downside is that ranges have to be specified and the tool can report absorption pesimistically if during the run-time the extreme ends of those ranges do not appear.

Floating point ``swamping'' is mentioned by Higham~\cite{high93s}.
In this work, different summation algorithms are tested and analysed when computing different series.
Swamping is mentioned for evaluating the Taylor series approximation of $\mathrm{e}^{-x} \approx \Sigma_{k = 0}^{N} (-x)^k/k!$ at $x = 2\pi$ in single precision.
When evaluated in this original ordering, the terms are increasing in magnitude for the first $7$ terms before decreasing, and alternate in sign.
This evaluation provides an example of catastrophic cancellation.
If we were performing summation on data of the same sign, we would expect that sorting by ascending in magnitude would mitigate stagnation since the smaller terms are accumulated first.
However, for this example, sorting by descending magnitude reduces stagnation and this leads to a reduction in error.
In descending order for this case, cancellation occurs when adding together the large operands.
When the smaller operands are added, which would usually undergo stagnation, the accumulator is smaller in magnitude than any of the addends that produced it, due to cancellation, therefore stagnation does not occur.
Other summation orderings fail to reduce stagnation because the accumulation of smaller values is absorbed when adding the larger terms, and this is best prevented by performing summation in descending order.

Sakr et al.~\cite{swc19} define floating-point swamping in the context of low precision operations for neural network training.
Their definitions for partial and total swamping are similar to our own for stagnation, however they define swamping on a single operation.
In the context of network training, swamping is defined for performing vector dot products of random variables.
The dot product evaluated up to $n$ terms is the sum $s_n = \Sigma_{i=1}^n p_i$ where $p_i$ are the product terms.
The accumulator is assumed to have $m_{acc}$ mantissa bits, and likewise $m_p$ mantissa bits for each product term.
Partial swamping is defined to occur when $2^{m_{acc} - m_p}|p_i| < |\widehat{s_{i-1}}| \leq 2^{m_{acc}}|p_i|$ and total swamping is defined to occur when $|\widehat{s_{i-1}}| > 2^{m_{acc}}|p_i|$.
Here, partial swamping is when $p_i$ has between $1$ and $m_p-1$ bits shifted off when aligning significands in addition,
and total swamping is when all $m_p$ bits of $p_i$ are shifted off.
However, total swamping includes the case where $\fl(\widehat{s_{i-1}} + p_i)$ returns one of the addends changed only by rounding up, which is a special case that we have more precisely defined as \emph{near-total stagnation} since in that case the addition modifies the accumulator's value by rounding.
The product terms are all assumed to be statistically independent with zero mean and variance $\sigma_p^2$.
Based on this assumption, the dot product $\widehat{s_n} = p_1 + \mathellipsis + p_n$ should satisfy $\operatorname{Var}(\widehat{s_n}) = n\operatorname{Var}(p_i)$.
However the authors argue that if swamping is occurring in the computed dot product $\widehat{s_n}$ then instead $\operatorname{Var}(\widehat{s_n}) = \tilde{n}\operatorname{Var}(p_i)$ for some constant $\tilde{n} < n$.
The authors propose that the ratio $\operatorname{Var}(\widehat{s_n})/ \operatorname{Var}({s_n})$ 
can be used to identify the severity of swamping.
This is defined as the \textit{variance retention ratio} (VRR).
If the VRR is much less than $1$, then significantly fewer than $nm_p$ bits have contributed to the dot product.
The authors provide a formula to determine the VRR based on only $n, m_{acc}$ and $m_p$.
If we fix two of these three parameters, we can identify the values of the third parameter where the VRR is closest to $1$ to identify which values would be most suitable to avoid swamping.

Stagnation is mentioned by Blanchard, Higham, and Mary~\cite{bhm20}.
The authors observe stagnation when performing sums of large vectors with uniformly-distributed random numbers in low precision formats.
It is mentioned that when stagnation occurs in an entirely non-negative sum, the rounding errors are all negative which may be much closer to the worst-case error bound.
In general, when all the data is of the same sign, error from stagnation only accumulates because all the rounding errors apply in the same direction.
We can justify here that detecting stagnation in floating-point arithmetic programs will help to identify significant error accumulation.
Our tool in Julia for stagnation detection is designed to differentiate between absorptions of positive and negative addends.

Croci et al.~\cite{cfhm21} and Fasi and Mikaitis~\cite{fami21} discuss stochastic rounding and demonstrate its effect on stagnation.
Stochastic rounding applies rounding randomly rather than deterministically, with probabilities relative to $1$ minus the relative closeness of each adjacent floating point number.
Croci~et~al.~\cite{cfhm21} state that exact stochastic rounding is immune to stagnation because the probability of absorption for a non-zero addition is never $1$.
El~Arar~et~al.~\cite{effm25} show that SR with limited precision random numbers is still subject to stagnation, but alleviates the effects of it compared with full stagnation that occurs with RN.

Furthermore, Croci~et~al.~\cite[Thm. 6.4]{cfhm21} consider the arbitrary vector dot product $y = a^\mathrm{T} b$, evaluated as $\widehat{y}$ in floating point arithmetic.
When the dot product is evaluated using stochastic rounding, the expectation $\mathds{E}(\widehat{y}) = y$ holds regardless of summation ordering.
This is not true for deterministic rounding; it was shown by Higham~\cite{high93s} that summation ordering has effect on the error.

Both works \cite{cfhm21} and  \cite{fami21} discuss stochastic rounding preventing stagnation in ODE/PDE applications and show that for sufficiently small timestep, stagnation is guaranteed in round-to-nearest but prevented with stochastic rounding.
The former discusses the numerical solution of the heat equation in one dimension, the latter explores the numerical solution of the initial value problem $\dot{x}(t) = y(t)$, $\dot{y}(t) = -x(t)$ with initial conditions $x(0) = 1$, $y(0) = 0$ using the explicit Euler method.
The numerical solution is evaluated in several 16-bit floating-point formats, with both round-to-nearest and stochastic rounding, for a range of timestep values.
The closed-form solution to this ODE is the unit circle in the $xy$-plane.
If the value of the timestep $\Delta t$ is not sufficiently small, the numerical solution will drift outwards from the unit circle.
Reducing $\Delta t$ improves accuracy of the solution for both round-to-nearest and stochastic rounding modes.
However, if $\Delta t$ is below a given threshold, stagnation will occur when using round-to-nearest, providing an octagon-shaped numerical solution to the unit circle.
This is an example where stagnation detection can help identify places of code that should be rewritten.

\section{Methods}

\subsection{Defining stagnation}

We have chosen that stagnation itself, being total, near-total or partial, refers to the number of absorptions taking place in a recursive summation.
The intuitive definition of stagnation restricts this to consecutive absorptions, but we have chosen to apply stagnation to cases beyond these.
From our testing, we have found no evidence that the common definition, referring to consecutive absorptions, is the superior definition.
Using our definition, we can measure the stagnation of any sum and report absorptions no matter the structure in which they occur.
Therefore we define stagnation of length $M$,
whether total, near-total or partial, as the number of absorptions of the corresponding type which have occurred in a summation of $n$ floating-point values, regardless of the position of absorptions in the sum.

\subsection{Reducing or eliminating stagnation}

Stagnation, when detected, can be reduced by various methods.
Stochastic rounding \cite{cfhm21} can be used to statistically preserve, through rounding, the quantities of the bits lost when aligning significands.
Compensated summation~\cite[Sec.~5.3.2]{mbdj18} can also be used, which computes the rounding errors and uses that information in later additions to reduce the error accumulation.
Lastly, since $u=p-d$ and $d$ depends on the exponents of the addends, higher precision $p$ can be used in parts of code that are subject to stagnation, to increase the average $u$.
Finally, stagnation can be reduced by reordering the sequence of operations if the data is available at the time of summation, aiming for additions to be of values with similar exponents~\cite{wolf64,grma25}---for example, if a set of positive numbers is sorted in a decreasing order it is well known that adding them in the reverse order reduces the error bound (whilst not necessarily the actual error~\cite{high93s}).

\subsection{Analysis of stagnation}

We can prove in what scenarios stagnation can and cannot happen which can help with testing the correctness of the automatic stagnation detection.

Take a set of positive numbers $x_1, \dots, x_n \in \F$ which are being generated during runtime, perhaps as part of a larger summation, and are not available for sorting or blocking in any way---summation is performed recursively.
We are approximating $s_n = \sum_{i=1}^{n}x_i$, whose approximation we denote by
\begin{equation}
  \label{eq:recursive}
  \widehat{s_n}=\fl(\widehat{s_{n-1}}+x_n)=\cdots\fl(\fl(s_1+x_2) + x_3) + \cdots,
\end{equation}
with $\widehat{s_1}=s_1=x_1$.
Here we prove various conditions in which stagnation can or cannot happen when computing \eqref{eq:recursive}.
We only look at positive values $x_i$ and the results apply analogously to the case of $x_i < 0$ for all $i \in [1, n]$.
If the sum contains addends of mixed signs, the analysis applies on the intermediate sums during which the signs of addends do not change.

\subsubsection{Stagnation and ordering of the addends}

\begin{theorem}
  Given $x_1, \dots, x_n \in \F$ where $x_i\geq 0$ for $i=1\dots n$ and a recursive sum $\widehat{s_n}=\fl(\widehat{s_{n-1}}+x_n)$ with $\widehat{s_1}=s_1=x_1$, if the sequence $x_2,\dots, x_n$ is decreasing or stays constant, total stagnation of length $1$ to $n-1$ can happen.
  \label{thm:stagn0}
\end{theorem}
\vspace{-0.5cm}
\begin{proof}
  The harmonic series is an example of this scenario.
  This is expressed as $1 + 1/2 + 1/3 + \mathellipsis = \sum_{i=1}^\infty 1/i$.
  In real arithmetic the summation diverges.
  In finite precision, the sum eventually grows to a point where the next addend is absorbed completely \cite{malo13}.
  The addends are decreasing, so total stagnation occurs.
  Another example is as follows. Take $x_1=2^{j}$ for $j \in \mathbb{Z}$ and $x_i<2^{j-p-1}$ for all $i=2,\dots, n$.
  Then $\widehat{s_n}=s_1=x_1$ since $\fl(s_1+x_i)=s_1$ for all $i=2\dots n$.
  The proof is analogous for shorter stagnations with $x_i<2^{j-p-1}$ for some $i=2 \dots n$.
\end{proof}

\begin{theorem}
  Given $x_1, \dots, x_n \in \F$ where $x_i\geq 0$ for $i=1\dots n$ and a recursive sum $\widehat{s_n}=\fl(\widehat{s_{n-1}}+x_n)$ with $\widehat{s_1}=s_1=x_1$, if the sequence $x_2,\dots, x_n$ is increasing, total stagnation of length $1$ to $n-1$ can happen.
    \label{thm:stagn1}
\end{theorem}
\begin{proof}
  Take $x_1=1$, $x_2=2^{p+1}$, $x_3=2^{(2p+2)}$ (thus, $x_i=2^{(i-1)(p+1)}$), then $\fl(x_1+x_2)=x_2$, $\fl(x_2+x_3)=x_3$ and so on, giving stagnation of length $n-1$. Stagnation of of length shorter than $n-1$ can be analogously shown.
\end{proof}

Theorem~\ref{thm:stagn1} shows an example of stagnation occuring from the left side, where the accumulator is changing on every step, but its value is one of the addends on that step.
This is a reverse scenario of a typical understanding of stagnation, where the accumulator is growing to the point where it starts to absorb the addends.
In this case, the accumulator is not growing enough and is absorbed by the following addends.

As shown by the previous two theorems, total stagnation can occur in both increasing and decreasing series.
We will consider how to constrain data such that total stagnation could not occur when performing recursive summation on a decreasing sequence.
First, we consider how to prevent stagnation given only a lower bound on the data.
This comes from the deduction that stagnation can occur in a sum of a descending sequence in two ways:
either the accumulator grows large enough that the next addends are rounded off completely,
or the addends themselves decrease by a significant enough amount to be rounded off in the next addition while the accumulator is relatively unchanged.
We can impose a lower bound on the data to prevent both possibilities.
\begin{theorem}
  Given a decreasing sequence of positive values $x_1, \mathellipsis, x_n \in \mathds{F}$ where $x_{k-1} > x_{k} > 0$ for $k = 2, \mathellipsis, n$
  and a recursive sum $\widehat{s_1} = x_1, \widehat{s_k} = \fl(\widehat{s_{k-1}} + x_k)$,
  total stagnation cannot occur if $x_k > x_1 / 2$ for all $k$ and $2^{p-1}x_1 < N_\mathrm{max}$ where $p$ is the precision of $\mathds{F}$ and $N_\mathrm{max}$ is the greatest normalised floating-point number.
  \label{thm:decsum-lb}
\end{theorem}
\vspace{-0.5cm}
\begin{proof}
  At step $k$ we are performing $\widehat{s_{k}} = \fl(\widehat{s_{k-1}} + x_{k})$.
  Absorption will not occur at step $k$ if $x_k > \ulp(s_{k-1})/2$, so we will show that this condition can hold provided the assumptions are held.

  First, we will prove the bound $\widehat{s_k} \leq kx_1$.
  We do this by induction.
  \begin{itemize}
    \item $\widehat{s_1} = x_1$ by definition.
    \item $\widehat{s_2} = \fl(\widehat{s_1} + x_2) \leq \fl(x_1 + x_1) = 2x_1$.
    \item Assume $\widehat{s_k} \leq kx_1$.
    \item Then $\widehat{s_{k+1}} = \fl(\widehat{s_k} + x_{k+1}) \leq \fl(kx_1 + x_1) = (k+1)x_1$, provided $k \leq 2^{p-1}$.
  \end{itemize}
  The greatest $n$ such that $x_k > x_1 / 2$ holds for $k = 1, \mathellipsis, n$ is $n = 2^{p-1}$.
  Let $n$ take this value and consider the sum at step $n$.
  The maximum value of the current addend is $x_n = x_{2^{p-1}} = x_1 / 2 + \ulp(x_1)$.
  Assume that the sequence of addends is decreasing at the smallest possible rate, where each subsequent addend is one $\ulp$ less than the previous, or half an $\ulp$ if the previous term is a power of two.
  The accumulator is bounded by $\widehat{s_{n-1}} \leq (n-1)x_1$, so $\ulp(\widehat{s_{n-1}})/2 \leq \ulp((2^{p-1}-1)x_1)/2 \leq 2^{p-2}\ulp(x_1) \leq x_1 / 2$.
  Therefore $x_n > \ulp(\widehat{s_{n-1}})/2$ holds.
  We have shown that if the given lower bound is satisfied, we will not observe stagnation for any recursive summation of decreasing floating point numbers.

  If the sequence $x_1, \mathellipsis, x_n$ decreases at a greater rate than the assumption, the number of terms $n$ when we reach $x_n = x_1 / 2 + \ulp(x_1)$ is less than $2^{p-1}$.
  We will not be able to hold the condition $x_k > x_1 / 2$ beyond this point, and stagnation may occur.
\end{proof}

The theorem above gives us a lower bound for a sequence of descending data based on the first argument.
If this bound is satisfied, we can guarantee prevention of total stagnation.
However, this is an overconservative bound which also imposes a limit on the number of terms that the sequence can contain.
Once a sequence falls below this bound, which can happen far before the asserted maximum number of terms in general, the result of this theorem does not hold, but stagnation still may not occur.
We will now consider a more general and simple result for summation of descending data.
\begin{proposition}
  Given a decreasing sequence $x_1, \mathellipsis, x_n \in \F$ where $x_i > x_{i+1} > 0$ for $i = 1, \mathellipsis, n$ and a recursive sum $\widehat{s_1} = x_1$, $\widehat{s_{k+1}} = \fl(\widehat{s_k} + x_{k+1})$ for $k = 2, \mathellipsis, n-1$,
  total stagnation cannot occur if $x_{k+1} > k \ulp(x_1)/2$ for $k = 2, \mathellipsis, n-1$.
  \label{prp:decsum-nostag}
\end{proposition}
\vspace{-0.5cm}
\begin{proof}
  At step $k+1$ we are adding $\widehat{s_k}$ and $x_{k+1}$. We have $\ulp(\widehat{s_k})/2 \leq \ulp(kx_1)/2 \leq k \ulp(x_1)/2$ by similar working as in Theorem \ref{thm:decsum-lb}, so we can write the condition for no total stagnation as $x_{k+1} > k \ulp(x_1)/2$.
\end{proof}
This is a more general result that depends on both the initial value of the sequence and the number of terms.

These theorems provide an understanding of how stagnation can appear in a recursive summation.
If we can identify that some data exhibits a similar pattern to what is shown in a theorem, we can deduce if stagnation may occur.
Theorem \ref{thm:stagn0} shows that with decreasing data, total stagnation can occur if the accumulator grows sufficiently such that the addends eventually decrease enough to be absorbed.
Theorem \ref{thm:stagn1} considers that if the addends themselves are growing, the accumulator may instead be absorbed relative to the large addends, and the repeated instance of this is recognised as stagnation.
We then show a restriction in Theorem \ref{thm:decsum-lb}, where if we enforce a bound on decreasing data undergoing recursive summation then we can guarantee no stagnation.
The arguments of Theorem \ref{thm:decsum-lb} apply to Proposition \ref{prp:decsum-nostag}, but the bound here depends on the length of the summation so the growth can be greater.

\subsection{Implementing stagnation detection in MATLAB}

As a proof of concept, we have overloaded the built-in \texttt{plus(in1, in2)} function in MATLAB to perform stagnation detection, measuring the number of instances where total or near-total absorptions occur (total and near-total stagnation combined) for single precision addition operations.
Global variables $\mathtt{u}$, $\mathtt{additions}$ and $\mathtt{total\_absorptions}$ are updated each time \texttt{plus} is called.
During addition, the difference in the exponents of the operands is recorded and used to compute the number of bits from the significand of the smaller addend that contribute to the addition.
The global $\mathtt{u}$ stores the accumulation of these contributing bits for all additions.
If all the bits of the smaller value's significand are shifted off, then $\mathtt{u}$ is unchanged and $\mathtt{total\_absorptions}$ is incremented.
$\mathtt{additions}$ is incremented each time the operation is called.
On completion of a recursive summation, the user can check these values to identify how stagnation has affected the sum.
The $\mathtt{total\_absorptions}$ field records the stagnation length, and the level of stagnation determined as the average $\mathtt{u}$ per sum can be obtained as $\mathtt{u/additions}$.

\subsection{Implementing stagnation detection in Julia}

While MATLAB was used for initial prototyping, our focus is on implementing a usable tool for the Julia programming language \cite{bek17} which is becoming a popular choice for scientific computing.
Our tool \texttt{StagnationDetection.jl} v0.1 is available on GitHub\footnote{
  https://github.com/willwoolf/StagnationDetection.jl
}.

\subsubsection{Design of the Julia tool}

We have developed \texttt{StagnationDetection.jl} for performing stagnation detection in Julia programs, designed to require minimal modification of source code.
The functionality is similar to the MATLAB tool mentioned above but with a greater breadth of features.
\texttt{StagnationDetection.jl} can be customised to measure total, near-total, or partial stagnation depending on a customiseable threshold.

The tool provides new data types \texttt{Float64SD}, \texttt{Float32SD} and \texttt{Float16SD}, which are stagnation-detection-activating alternatives to the \texttt{Float64}, \texttt{Float32} and \texttt{Float16} floating point types respectively.
We will refer to these types as \textit{stagnation detection types}, belonging under the abstract type \texttt{FloatSD}.
When performing arithmetic using these data types, stagnation detection is performed automatically.

We will discuss the features of the \texttt{Float64SD} data structure specifically, so we can distinguish between \texttt{Float64SD} and the abstract type \texttt{FloatSD} it belongs to.
However all of \texttt{Float64SD}, \texttt{Float32SD} and \texttt{Float16SD} offer the same features.
A \texttt{Float64SD} stores six fields, named \texttt{value}, \texttt{u}, \texttt{t}, \texttt{additions}, \texttt{absorptions\_up} and \texttt{absorptions\_dn}.
\begin{itemize}
  \item \texttt{value} (either \texttt{Float64}, \texttt{Float32} or \texttt{Float16}) stores the value for regular floating-point arithmetic.
  \item \texttt{u} (\texttt{UInt64}) is the number of bits preserved across every addition that has contributed to this variable.
  If $x = x_1 + \mathellipsis + x_n$ with stagnation detection then \texttt{u} is $u_1 + \mathellipsis + u_{n-1}$.
  \item \texttt{t} (\texttt{Int8}) is the threshold for the bits preserved in an addition, thus is used to determine whether we are measuring total ($t = -1$), near total ($t = 0$) or partial ($t > 0$) stagnation. This will be reflected by reporting functions.
  \item \texttt{additions} (\texttt{UInt64}) is the number of additions that have been performed in contributing to this variable.
  \item \texttt{absorptions\_up} (\texttt{UInt64}) is the number of additions that have been classified as an absorption, depending on the operands and $t$, where the rounded-off operand is negative and thus the rounding error is positive.
  \item \texttt{absorptions\_dn} (\texttt{UInt64}) is the number of absorptions in the opposite direction to \texttt{absorptions\_up}.
\end{itemize}

In Julia syntax we write \texttt{f(x::T, y::U)} to mean a function on two arguments, \texttt{x} of type \texttt{T} and \texttt{y} of type \texttt{U}.
Stagnation detection is performed automatically when calling \texttt{+(in1::T, in2)} or \texttt{+(in1, in2::T)} where \texttt{T} is a stagnation detection type.
First we measure $u_i$ as the number of bits preserved at the current addition, computed as $\max(0, p - |e_a - e_b|)$ where $e_a$, $e_b$ are the exponents of the addends.
Next, we determine if absorption is taking place.
We check $u_i$ relative to the absorption threshold $t$.
If we are checking for total stagnation, the criterion for absorption is stronger than checking $u_i$ is equal to zero.
We check, assuming we are adding $x, y \in \mathds{F}$ where $|x| < |y|$, that $|y| < |\ulp(x)|/2$ to guarantee total absorption.
Afterwards, we compute the appropriate updates to \texttt{absorptions\_[up,dn]}.
Operator function \texttt{+} returns a new \texttt{T} with the updated fields.

Since stagnation detection is automatic and uses a specialised method for the \texttt{+} function, it requires little modification for Julia programs to apply the tool.
For example, by specialising a variable as \texttt{x::Float64SD} instead of \texttt{x::Float64}, stagnation detection will apply for any addition operations that use this or subsequent variables.
When both addends are \texttt{FloatSD}, we have chosen the convention that only the stagnation data from the \emph{left hand side} operand will be copied and updated.
If only the left or right argument is a \texttt{FloatSD}, the data from this argument is copied directly.
When the floating-point addition operation with stagnation detection takes place, the \texttt{u}, \texttt{additions} and \texttt{absorptions\_up}/\texttt{dn} fields are updated to include the data from this addition.

We have provided several user-facing ``diagnostic'' functions.
The simplest is \texttt{report(::FloatSD)}, which prints a summary of the stagnation data in a format that should be readable to the average user.
There is also \texttt{reportall()}, which returns extremely short summaries for all \texttt{FloatSD}s as well as arrays containing \texttt{FloatSD} in the local scope.
We have not found a way for \texttt{reportall()} to work within a function, so it may only be called from the REPL.
We have also provided the following number-valued functions:
\begin{itemize}
  \item \texttt{num\_additions(::FloatSD)} (return \texttt{additions}).
  \item \texttt{num\_absorptions(::FloatSD)} (return \texttt{absorptions\_up + absorptions\_dn}).
  \item \texttt{diff\_absorptions(::FloatSD)} (absolute difference between positive and negative absorptions, unsigned).
  \item \texttt{net\_absorptions(::FloatSD)} (difference between positive and negative absorptions, signed).
\end{itemize}
These functions are provided for compact stagnation analysis. 

We have provided implementations of core mathematical functions in Julia that operate on stagnation detection types the same as they do for real types.
These mathematical functions are \texttt{sin}, \texttt{cos}, \texttt{sinpi}, \texttt{cospi}, \texttt{tan}, \texttt{sinh}, \texttt{cosh}, \texttt{tanh}, \texttt{exp}, \texttt{log}, \texttt{sqrt}, \texttt{abs}, \texttt{abs2}, \texttt{acos}, \texttt{asin}, \texttt{atan}, \texttt{abs}, \texttt{hypot}, \texttt{atan} and \texttt{rem}.
When using \texttt{abs}, stagnation detection data of the input is copied to the output.
Stagnation data cannot be preserved when using any of the other functions mentioned.

\subsubsection{Harmonic series with stagnation detection}

We provide an example for how a user can apply stagnation detection to their Julia code.
Consider the following function, which computes the harmonic series (see proof of Theorem \ref{thm:stagn0} earlier) in format \texttt{T} up to $n$ terms:
\begin{lstlisting}[language=Julia]
function harm(n::Int, T::Type = Float64)
    acc::T = 0.0
    for k in 1:n
        acc += one(T)/k
    end
    return acc
end
\end{lstlisting}
The addends are written as \texttt{one(T)/k} since \texttt{1/k} instead would default to \texttt{Float64} and trigger type promotion for other accumulator types.
To apply stagnation detection, we can simply add \texttt{StagnationDetection} to the preamble, and call \texttt{harm} with our stagnation detection type of choice.
\begin{lstlisting}[language=Julia]
using StagnationDetection

h = harm(4_000_000, Float32SD)
\end{lstlisting}
We can then use \texttt{report} to generate a summary for the user.
\begin{lstlisting}[language=Julia]
report(h)
\end{lstlisting}
\begin{verbatim}
--------------- SUMMARY ----------------
Detecting total stagnation in 4000000 additions.
The condition for total absorption is u = 0 and for the smaller addend
to be rounded off completely in a given addition.
Accumulator precision: 24 (Float32SD).

Net stagnation length 1902848 (round down).
1902848 negative, 0 positive.
Overall stagnation length 1902848 if ignoring cancellation of rounding
directions.

-------- PROCESSING OF RESULTS ---------
u is the total number of bits which were not shifted off in the
addition.
Total number of bits processed u = 2098860.
Average u_i = 0.524715.
See paper: Analysis and Detection of Stagnation in Floating-Point
Summation,
Mantas Mikaitis and William Woolfenden, 2026.

------------ END OF SUMMARY ------------
\end{verbatim}

We discussed stagnation in the harmonic series in Theorem \ref{thm:stagn0}.
By using \texttt{report} we are told how many additions have taken place, as well as how many of those additions count towards stagnation.
Stagnation is reported by all instances, but also by the difference of positively and negatively rounded absorptions, assuming that single instances of absorptions in in opposite rounding directions effectively cancel out.
We are also given the average level of stagnation based on the accumulation of preserved bits \texttt{u} across all additions.
Almost all the output of \texttt{report} is dependent on the properties of the argument, including which kind of stagnation (total, near-total or partial) is being measured.

\subsubsection{Matrix multiplication with stagnation detection}

We also show how stagnation detection can be applied in matrix multiplication, and discuss different ways to fetch information on stagnation.
Matrix multiplication works because the operations \texttt{+(::FloatSD, ::FloatSD)} and \texttt{*(::FloatSD, ::FloatSD)} are defined, which are called by matrix multiply.
We can perform large rectangular matrix multiplication with stagnation detection using the following script:
\begin{lstlisting}[language=Julia]
n = 5
m = 1_000_000

rng = Xoshiro(1)

A = rand(rng, Float32SD, n, m) # uniform distribution 
B = randn(rng, Float32SD, m, n) # normal distribution

C = A * B
\end{lstlisting}
and may then use any of the diagnostic functions for analysis. For example:
\begin{lstlisting}
net_absorptions.(C)
\end{lstlisting}
\begin{verbatim}
5x5 Matrix{Int64}:
 -2    5    8  -2   -6
 -1   -3    0   6  -25
 11  -26   -6  -8    5
  1   -8   -5   6    3
 18    4  -14  -5   14
\end{verbatim}
Each of \texttt{num\_additions}, \texttt{num\_absorptions}, \texttt{diff\_absorptions} and \texttt{net\_absorptions} can be broadcast to apply element-wise to the matrix.
We can also apply \texttt{reportall} to the current example.
\begin{lstlisting}[language=Julia]
reportall()
\end{lstlisting}
\begin{verbatim}
---------  SCALARS IN GLOBAL SCOPE  --------- 

 ----------- ARRAYS IN GLOBAL SCOPE  --------- 
A (Matrix{Float32SD}, 5x1000000), stagnation range 0.0% to 0.0%, mean
0.0%
B (Matrix{Float32SD}, 1000000x5), stagnation range 0.0% to 0.0%, mean
0.0%
C (Matrix{Float32SD}, 5x5), stagnation range 0.003% to 0.018%, mean
0.011%

 -------------- END OF REPORTALL  ------------
\end{verbatim}
There are no scalar variables in the global scope, so none are shown.
For \texttt{C}, \texttt{reportall} uses \texttt{diff\_absorptions} to obtain the minimum, mean and maximum values in the range of stagnation, being the proportion of additions that count towards stagnation.
The matrices \texttt{A} and \texttt{B} are constructed with \texttt{Float32SD} arguments thus show no history with regard to stagnation.
While we used \texttt{net\_absorptions} in the example above to show mixed sign in the stagnation counts, \texttt{reportall} uses \texttt{diff\_absorptions} to identify stagnation without regard to sign.

\section{Results}

\subsection{MATLAB stagnation detection}

\begin{table}[!b]
  \centering
  \caption{Average $u$ in various sums using binary32 arithmetic.}
  \begin{tabular}{lllll}
    \toprule
    & $i=0.5\times 10^6$ & $10^6$ & $1.5\times 10^6$ & $2.5\times 10^6$ \\
    \midrule
    H & 3.05 & 2.05 & 1.39 & 0.83 \\
    IH & 6.39 & 5.7 & 5.18 & 5.01 \\
    R1 & 6.05 & 5.06 & 4.43 & 3.73  \\
    R2 & 22.66 & 22.7 & 22.7 & 22.87 \\
    P & 0.019 & 0.009 & 0.006 & 0.0039 \\
    \bottomrule
  \end{tabular}
  \label{table:average-u}
\end{table}

\begin{table}[!b]
  \centering
  \caption{Relative error and stagnation length in various sums using binary32 arithmetic and $n = 2.5 \times 10^6$ iterations.}
  \begin{tabular}{lllllll}
    \toprule
    & Rel. error (RN) & Stagnation (RN) & Rel. Error (SR) \\
    \midrule
    H & 0.0062 & 1451424 & 7.7742e-05 \\
    IH & 2.7101e-05 & 0 & 3.7974e-06  \\
    R1 & 1.7287e-05 & 138203  & 2.4790e-05 \\
    R2 & 1.4511e-04 & 0 & 5.4876e-06 \\
    P & 1.2666e-04 & 2497104 & 4.7314e-06 \\
    \bottomrule
  \end{tabular}
  \label{table:experiment-stagn}
\end{table}

We test our MATLAB tool on some experiments with individual summations.
To test the progression of $u$ for some basic algorithms, we used the harmonic sum (H): $\sum_{i=1}^{n}1/i$ (which is known to stagnate \cite{malo13,hipr19} as $i \to \infty$); the harmonic sum computed in reverse order starting from $n$ (IH); a sum of random numbers $x_i \in (0, 1)$ (R1); a modified sum of random numbers that resets to zero every time it reaches $1$ (R2); and the sum that converges to $\frac{\pi^2}{6}$ (P): $\sum_{i=1}^{n}i^{-2}$.
For each sum, we are measuring the level of stagnation by computing \texttt{u/additions} as well as the length of stagnation measured in \texttt{total\_absorptions}.
Table~\ref{table:average-u} demonstrates the average $u$ at different points of each running sum for $n=2.5\times 10^6$.
Table~\ref{table:experiment-stagn} additionally shows the combined length of total and near-total stagnation and the relative errors compared with the sum evaluated in binary64 for round-to-nearest ties-to-even (RN)~\cite{ieee19} and stochastic rounding (SR)~\cite{cfhm21} which is known to alleviate stagnation errors.
We have simulated SR with CPFloat~\cite{fami23} which uses 32-bit random number generator for implementing SR.

As expected, problem $P$ has the lowest average $u$, meaning most information was lost, which is expected as most addends are very small due to $i^2$ in the denominator.
The harmonic sum (H) is the second worst, which is also expected since the sum starts with a relatively big quantity of $1$ and is then growing whilst the addends are decreasing.
IH and R1 are very similar, with an average of around $5$ bits being used in the additions.
The R2 problem almost does not have any information loss, with approximately all 23 bits used in additions---this is because we reset the accumulator every time it reaches $1$.

Here we can come back to the parameter $t$, which is a threshold for partial stagnation.
For the purpose of an example, set $t=2$.
In the results of Table~\ref{table:average-u} we can then declare that problem $H$ partially stagnated after $i=10^6$ steps, while problem $P$ stagnated very early in the evaluation.
However, problems IH and R1 also have a decreasing average $u$, which signals that most likely, for bigger $n$, they would also start partially stagnating.
The problem R2 is not looking like it is going to stagnate at all as the average $u$ remains the same irrespective of $i$, as is expected because of accumulation restart.

Table~\ref{table:experiment-stagn} shows that for $H$ and $P$ sums, SR recovers the errors caused by stagnation with the RN mode.
These two problems had more than $50\%$ of all additions take part in total or near-total stagnation and users or precision tuners here could decide to apply SR.
Whilst $IH$ and $R2$ have no stagnation, the error is reduced by SR due to it having the property of $\sqrt{n}$ error growth~\cite{chm21}.
While $R1$ has a small proportion of stagnation compared with $n$, SR does not appear to reduce the error.

Theorem \ref{thm:stagn0} tells us that stagnation can occur if the addends are decreasing or remain constant, which is reflected by the data in the sums (H) and (P).
While in real arithmetic (P) converges but (H) doesn't, in floating-point both involve decreasing data that leads to stagnation.
The (IH) sum performs ascending order summation, but theorem \ref{thm:stagn1} does not apply because the difference in magnitude between addends is not great enough for stagnation to occur.
None of our theory concerns the random sums (R1) and (R2), which involve summation on data sorted randomly.

\subsection{Julia stagnation detection}

All computations in this section\footnote{
  https://github.com/willwoolf/stagnationtesting
} were performed on a Dell Latitude 7450 computer with an Intel Core Ultra 7 165H and 32GB of LPDDR5X memory.
Since stagnation detection is intended to be a low cost solution for detecting significant accumulation of floating-point rounding errors, we have collected data on a standard computer to reflect the intended use case.
The same applies for the benchmarks discussed in the section afterwards.

\subsubsection{Verification of conditions for stagnation}

Theorem \ref{thm:decsum-lb} provides a bound on a sequence in order to guarantee that total stagnation will not occur.
We have verified this with numerical experiments, using \texttt{StagnationDetection.jl}.
We start by constructing a vector of summands.
The most significant case is the vector starting with $x>0$ in the first entry, and descending by consecutive floating-point number until $\operatorname{nextfloat}(x/2)$.
If adjacent values had a larger difference, there would be fewer terms in the sum.
If negative values were present in the sum, stagnation would be less likely since this would oppose growth of the accumulator.
We have tested this in single and half-precision for the greatest possible range of values and have observed that total stagnation does not occur, in agreement with the theorem.

\subsubsection{Summation of random numbers}

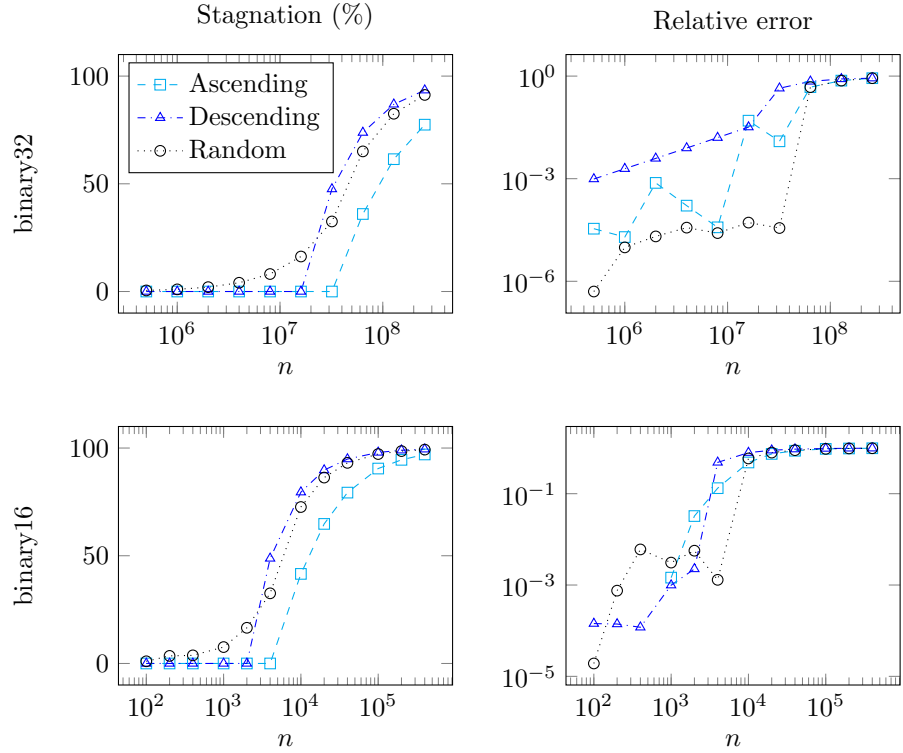
\begin{figure}[b!]
  \centering
  \begin{tikzpicture}
  \begin{groupplot}[
    group style={
      group size=2 by 2,
      horizontal sep=1.5cm,
      vertical sep=1.5cm,
    },
    width=6cm,
    height=5cm,
    legend pos = {north west},
    legend style={legend cell align = {left}}
    ]

    \nextgroupplot[
      title = {Stagnation ($\%$)},
      xmode = {log},
      xlabel = {$n$},
      ymin = 0,
      ymax = 100,
      ylabel={binary32},
      enlarge y limits = {0.1}
    ]
    
    \addplot[color = cyan, dashed, mark = square, mark options={solid}]
    table[
    x = {n},
    y = {stagf(asc)},
    col sep = comma
    ]{stagu32.dat};
    \addlegendentry{Ascending}

    \addplot[color = blue, dashdotted, mark = triangle, mark options={solid}]
    table[
      x = {n},
      y = {stagf(desc)},
      col sep = comma,      
    ]{stagu32.dat};
    \addlegendentry{Descending}

    \addplot[color = black, dotted, mark = o, mark options={solid}]
    table[
      x = {n},
      y = {stagf(rand)},
      col sep = comma
    ]{stagu32.dat};
    \addlegendentry{Random}

    \nextgroupplot[
      title = {Relative error},
      xmode = {log},
      ymode = {log},
      xlabel = {$n$},
      ymin = 0,
      ymax = 1,
      enlarge y limits = {0.1}
    ]

    \addplot[color = cyan, dashed, mark = square, mark options={solid}]
    table[
      x = {n},
      y = {relerr(asc)},
      col sep = comma
    ]{stagu32.dat};

    \addplot[color = blue, dashdotted, mark = triangle, mark options={solid}]
    table[
      x = {n},
      y = {relerr(desc)},
      col sep = comma
    ]{stagu32.dat};

    \addplot[color = black, dotted, mark = o, mark options={solid}]
    table[
      x = {n},
      y = {relerr(rand)},
      col sep = comma
    ]{stagu32.dat};

    \nextgroupplot[
      title = {},
      xmode = {log},
      xlabel = {$n$},
      ymin = 0,
      ymax = 100,
      ylabel={binary16},
      enlarge y limits = {0.1}
    ]

    \addplot[color = cyan, dashed, mark = square, mark options={solid}]
    table[
      x = {n},
      y = {stagf(asc)},
      col sep = comma
    ]{stagu16.dat};

    \addplot[color = blue, dashdotted, mark = triangle, mark options={solid}]
    table[
      x = {n},
      y = {stagf(desc)},
      col sep = comma
    ]{stagu16.dat};

    \addplot[color = black, dotted, mark = o, mark options={solid}]
    table[
      x = {n},
      y = {stagf(rand)},
      col sep = comma
    ]{stagu16.dat};

    \nextgroupplot[
      title = {},
      xmode = {log},
      ymode = {log},
      xlabel = {$n$},
      ymin = 0,
      ymax = 1,
      enlarge y limits = {0.1}
    ]

    \addplot[color = cyan, dashed, mark = square, mark options={solid}]
    table[
      x = {n},
      y = {relerr(asc)},
      col sep = comma
    ]{stagu16.dat};

    \addplot[color = blue, dashdotted, mark = triangle, mark options={solid}]
    table[
      x = {n},
      y = {relerr(desc)},
      col sep = comma
    ]{stagu16.dat};

    \addplot[color = black, dotted, mark = o, mark options={solid}]
    table[
      x = {n},
      y = {relerr(rand)},
      col sep = comma
    ]{stagu16.dat};
  \end{groupplot}
  \end{tikzpicture}
  \caption{
    Stagnation and relative error for vector summation of the uniform $U[0, 1]$ distribution.
    Recursive summation is performed in ascending, descending and random ordering in both single (binary32) and half precision (binary16).
    Summation is performed once per ordering on randomly generated data.
  }
  \label{fig:uniformvecsums}
\end{figure}

We perform recursive summation of randomly generated vectors with entries sampled from the uniform distribution over $[0, 1]$ in \texttt{binary32} single precision and \texttt{binary16} half precision.
In theory, we would expect stagnation to be minimal for summation in ascending ordering, since values of smallest magnitude will be accumulated earliest.
In all cases, sorting is performed by absolute value.
Relative error is computed in reference to the same summation performed in \texttt{binary64} double precision.

We performed summation up to length $N = 2.56 \times 10^8$ in single precision and up to $N = 4 \times 10^5$ in half precision.
Recursive summation is performed on the vector of length $N$, interrupting the summation at given lengths $n<N$ to measure stagnation and relative error.

We provide results for vector summations in Figure \ref{fig:uniformvecsums}.
For both precisions tested, we observe that stagnation is consistently minimal in ascending ordering.
With ascending and descending ordering, stagnation is prevented until a point in the summation where the accumulator has grown significantly, and stagnation will occupy most if not all of the remaining summation.
This occurs in descending order for smaller $n$ than in ascending, since larger values are accumulated earlier in descending ordering and not all data has been accumulated.
Once this point is reached in descending order summation, the growth of stagnation is greater than any other ordering since stagnation will occur for the entire remaining sum.
After this point, stagnation in descending order is greater than in random.
In random ordering, growth of stagnation is linear and non-zero for much smaller $n$ than the ordered summations.
The growth of the accumulator can be expected to be linear, so for a given $n$ we can estimate how small the next value from $U[0, 1]$ must be in order to be absorbed in the next addition.

Relative error in the summation grows sharply and reaches a plateau near $1$ for all orderings after $n \approx 6.4\times 10^7$ in single precision and $n \approx 10^4$ in half precision.
This corresponds mostly to regions where stagnation is reported to occupy greater than $50\%$ of the summation.
We can clearly identify from the figures that for descending and random ordering, there is a link between high relative error and high stagnation.
When stagnation occupies such a large proportion of the sum, such as $50\%$ or more, this means the information being lost due to stagnation can be comparable to the size of the accumulator itself.
As this proportion of stagnation grows, more values are rounded off and the accumulator stagnates while the actual value of the sum grows, increasing error.
We observe that this occurs in all orderings, but the reported rate of stagnation in ascending ordering is consistently lower than that of descending and random.
However, stagnation in ascending ordering corresponds to the largest values in the sum being absorbed, introducing the greatest possible error per absorption.

For values smaller than $n \approx 6.4\times 10^7$ in single precision, where the plateau in relative error occurs, the effectiveness of stagnation detection as an indication of error is not clear.
In single precision, random ordering has consistently the best or tied-best error despite being the only ordering that reports some stagnation for this range of $n$.
We could deduce that for summations where stagnation occupies a proportion of the sum below a certain threshold, error may occur from other sources and stagnation detection may not be useful at indicating error.
In half precision there appears to be no best ordering for reducing relative error for all $n$.
However, for $n = 100$, $200$ and $400$, ascending order summation in half precision attains zero error, which is not visible on the logarithmic axis.
When stagnation is not significantly high, the patterns in relative error may also be attributed to the randomness of the data and subseqently the summation.

Stagnation in descending order summation for uniform distribution data follows the idea of Theorem \ref{thm:stagn0}.
As the addends are decreasing and the accumulator is growing, stagnation will occur after a sufficient number of additions.
In this case stagnation will continue as the addends will continue to decrease in magnitude.
We also tested stagnation detection on vector summation of values from the normal distribution (Gaussian) with mean $\mu = 0$ and variance $\sigma^2 = 1$, however we found that summation on mixed sign data leads to no or very little stagnation.
Stagnation is not impossible for this summation data, but in our testing it is consistently much less than $1\%$ of the summation and therefore insignificant.
None of our theory on preventing stagnation (Theorem \ref{thm:decsum-lb}, Proposition \ref{prp:decsum-nostag}) applies, and neither theorems \ref{thm:stagn0} or \ref{thm:stagn1} apply to mixed sign data.

\subsubsection{Inner products}

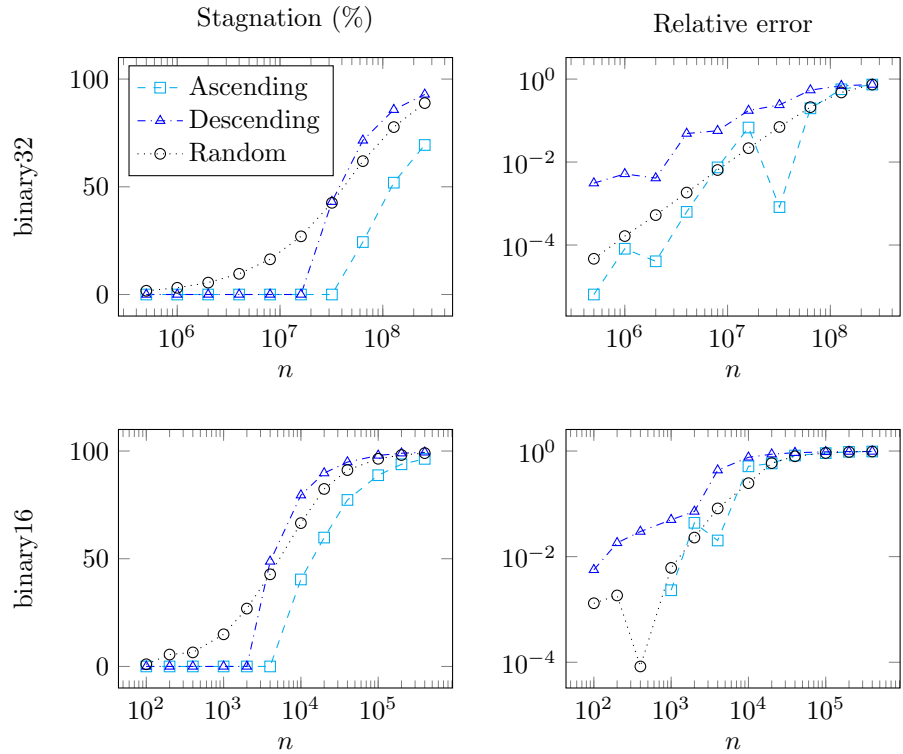
\begin{figure}[b!]
  \centering
  \begin{tikzpicture}
  \begin{groupplot}[
    group style={
      group size=2 by 2,
      horizontal sep=1.5cm,
      vertical sep=1.5cm,
    },
    width=6cm,
    height=5cm,
    xlabel={},
    legend pos = {north west},
    legend style={legend cell align = {left}}
    ]

    \nextgroupplot[
      title = {Stagnation ($\%$)},
      xmode = {log},
      xlabel = {$n$},
      ymin = 0,
      ymax = 100,
      ylabel={binary32},
      enlarge y limits = {0.1}
    ]
    
    \addplot[color = cyan, dashed, mark = square, mark options={solid}]
    table[
    x = {n},
    y = {stagf(asc)},
    col sep = comma
    ]{staguu32.dat};
    \addlegendentry{Ascending}

    \addplot[color = blue, dashdotted, mark = triangle, mark options={solid}]
    table[
      x = {n},
      y = {stagf(desc)},
      col sep = comma,      
    ]{staguu32.dat};
    \addlegendentry{Descending}

    \addplot[color = black, dotted, mark = o, mark options={solid}]
    table[
      x = {n},
      y = {stagf(rand)},
      col sep = comma
    ]{staguu32.dat};
    \addlegendentry{Random}

    \nextgroupplot[
      title = {Relative error},
      xmode = {log},
      ymode = {log},
      xlabel = {$n$},
      ymin = 0,
      ymax = 1,
      enlarge y limits = {0.1}
    ]

    \addplot[color = cyan, dashed, mark = square, mark options={solid}]
    table[
      x = {n},
      y = {relerr(asc)},
      col sep = comma
    ]{staguu32.dat};

    \addplot[color = blue, dashdotted, mark = triangle, mark options={solid}]
    table[
      x = {n},
      y = {relerr(desc)},
      col sep = comma
    ]{staguu32.dat};

    \addplot[color = black, dotted, mark = o, mark options={solid}]
    table[
      x = {n},
      y = {relerr(rand)},
      col sep = comma
    ]{staguu32.dat};

    \nextgroupplot[
      title = {},
      xmode = {log},
      xlabel = {$n$},
      ymin = 0,
      ymax = 100,
      ylabel={binary16},
      enlarge y limits = {0.1}
    ]

    \addplot[color = cyan, dashed, mark = square, mark options={solid}]
    table[
      x = {n},
      y = {stagf(asc)},
      col sep = comma
    ]{staguu16.dat};

    \addplot[color = blue, dashdotted, mark = triangle, mark options={solid}]
    table[
      x = {n},
      y = {stagf(desc)},
      col sep = comma
    ]{staguu16.dat};

    \addplot[color = black, dotted, mark = o, mark options={solid}]
    table[
      x = {n},
      y = {stagf(rand)},
      col sep = comma
    ]{staguu16.dat};

    \nextgroupplot[
      title = {},
      xmode = {log},
      ymode = {log},
      xlabel = {$n$},
      ymin = 0,
      ymax = 1,
      enlarge y limits = {0.1}
    ]

    \addplot[color = cyan, dashed, mark = square, mark options={solid}]
    table[
      x = {n},
      y = {relerr(asc)},
      col sep = comma
    ]{staguu16.dat};

    \addplot[color = blue, dashdotted, mark = triangle, mark options={solid}]
    table[
      x = {n},
      y = {relerr(desc)},
      col sep = comma
    ]{staguu16.dat};

    \addplot[color = black, dotted, mark = o, mark options={solid}]
    table[
      x = {n},
      y = {relerr(rand)},
      col sep = comma
    ]{staguu16.dat};
  \end{groupplot}
  \end{tikzpicture}
  \caption{
    Stagnation and relative error for vector dot products of random vectors from the uniform $U[0, 1]$ distribution.
    The two input vectors are multiplied together element-wise to obtain a single vector to perform summation on.
    Recursive summation is performed in ascending, descending and random ordering in both single (binary32) and half precision (binary16).
    Summation is performed once per ordering.
  }
  \label{fig:uniformdots}
\end{figure}

We extend the previous example to test dot products of uniformly-distributed random vectors, shown in Figure \ref{fig:uniformdots}.
Most of our analysis of the equivalent vector summation carries over, such as the rate of growth of stagnation in different orderings.
This result is effectively a vector sum of a different random distribution obtained by taking the product of two identical random variables on $U[0,1]$.
Differences include the rate of growth of relative error in single precision with random ordering, which obeys a visible linear progression.
Descending order sorting leads to greater stagnation than was observed with the equivalent vector sums, in both single and half precision.
However overall trends of stagnation and of relative error are very similar, such as where half precision ascending summation may attain zero error for small $n$, and the regions where large stagnation appears to correlate with large relative error.

\subsubsection{Matrix-matrix multiplication}

\begin{table}[b!]
  \centering
  \caption{
    Stagnation (mean) observed in matrix multiplication of $(A \in \mathds{F}^{5 \times n}) \times (B \in \mathds{F}^{n \times 5})$ uniformly distributed random matrices in single precision.
  }
  \begin{tabular}{llllll}
    \toprule
    Summation length & $n = 2^{22} \approx 4 \times 10^6$ & $2^{23}$  & $2^{24}$ & $2^{25}$ & $2^{26}$ \\
    \midrule
    Stagnation & $9.8\%$      & $16.7\%$  & $27.5\%$ & $43.3\%$ & $63.0\%$ \\
    \bottomrule
  \end{tabular}
  \label{tab:stagmatmul32}
\end{table}

\begin{table}[b!]
  \centering
  \caption{
    Stagnation (mean) observed in matrix multiplication of $(A \in \mathds{F}^{5 \times n}) \times (B \in \mathds{F}^{n \times 5})$ uniformly distributed random matrices in half precision.
  }
  \begin{tabular}{llllll}
    \toprule
    Summation length & $n = 2^9 = 512$ & $1024$   & $2048$   & $4096$   & $8192$ \\
    \midrule
    Stagnation & $9.4\%$   & $16.7\%$ & $27.6\%$ & $43.3\%$ & $62.7\%$ \\
    \bottomrule
  \end{tabular}
  \label{tab:stagmatmul16}
\end{table}

For the examples on vector summations and dot products, we produced vectors up to in the magnitude of $10^8$.
Large matrix multiplication with stagnation detection is challenging due to its computational cost.
Instead here we show rectangular matrix multiplication with one dimension fixed by performing multiplication of $A \in \mathds{R}^{m \times n}$ and $B \in \mathds{R}^{n \times m}$ where $n$ is large but $m$ is small.
We pick $m = 5$ and values $n$ from $2^{p-2}$ to $2^{p+2}$, where $p$ is the precision of the format being used.
While this leads us to pick different matrix sizes for \texttt{Float16} and \texttt{Float32} with no overlap, we observe that between both formats the rate of stagnation depends on the length of the summation relative to the precision.
Tables \ref{tab:stagmatmul32} and \ref{tab:stagmatmul16} show the stagnation for matrix multiplication of these dimensions.
Each example is equivalent to computing $m \times m$ dot products of length $n$.
The rate of growth of stagnation obtained from these results is less than linear.
This is because for a recursive sum on uniform distribution data, the probability of absorption in subsequent additions is increasing with the accumulator but this does not affect the previous additions.

\subsection{Shallow water simulation stagnation detection}

\subsubsection{Implementation choice and approach}

{ShallowWaters.jl} \cite{kdp20} is a Julia package for shallow fluid dynamics, which can be used for climate modelling.
We have decided to test stagnation detection when using ShallowWaters, to determine whether stagnation is a meaningful metric for floating-point reliability in a type-flexible scientific computation that performs time integration.
The default parameters for ShallowWaters define a fluid layer in a $4000 \mathrm{km} \times 2000 \mathrm{km}$ domain, $500\mathrm{m}$ thick at rest, subject to boundary conditions, over a $200$ day timespan.
We run this simulation on types \texttt{Float32SD} and \texttt{Float16SD} to measure stagnation, and compare against \texttt{Float64} to measure relative error, while all other options are constant.
The simulation returns the sea surface temperature \texttt{sst}, the horizontal and vertical current velocities \texttt{u}, \texttt{v} and the layer height {$\eta$}, which are referred to as the \textit{prognostic variables}.
We measure stagnation and root-mean-squared-error (RMSE) for the result of the computation of layer height $\eta$, comparing the result to an equivalent computation performed with entirely \texttt{Float64} arithmetic.
This approach is used in the developers' own work to measure error \cite{kdp20}.
We choose to not measure element-wise relative error of the result, as we have inspected relative error measures, and have found that the varying dynamic range of the fluid layer height in the simulation leads to relative error being represented in a matrix with extremely large ($\gg 1$) and small values.
We did not observe significant stagnation in the computation of the other prognostic variables in our tests, but this could be investigated by future work.

As mentioned, our motivation for testing with ShallowWaters is due to its flexibility for operating with different floating-point types.
Other packages for fluid simulation and climate modelling in Julia are \texttt{Oceananigans.jl} \cite{Oceananigans-overview-paper-2025} and \texttt{SpeedyWeather.jl} \cite{SpeedyWeatherJOSS} among others.
Groups such as \texttt{CliMA}, \texttt{JuliaClimate} and \texttt{JuliaOcean} are responsible for these software packages.
\texttt{SpeedyWeather.jl} shares a lot in common with ShallowWaters, namely type flexibility due to a focus on performance.
Oceananigans is designed with a focus on GPU arithmetic, which is beyond the scope of our current work.
Overall, while we are focusing on ShallowWaters we cannot provide any reason as to why these other tools would not yield results from stagnation detection.

\subsubsection{Configuration of the shallow water model}

We have run the default model simulation provided in \texttt{ShallowWaters.jl} with \texttt{Float16} arithmetic with stagnation detection.
However we have made several changes for our testing.
For performing stagnation detection, we set the type \texttt{Tprog} to \texttt{Float16SD} and stagnation detection is performed automatically for the \textit{prognostic variables} \texttt{u}, \texttt{v}, \texttt{$\eta$} and \texttt{sst}.
The simulation integrates the change in the prognostic variables within a $(\mathtt{nx} = 100) \times (\mathtt{nx}/2 = 50)$ matrix by default, used to represent the $4000 \mathrm{km} \times 2000 \mathrm{km}$ domain.
Usually, \texttt{ShallowWaters.jl} applies a scaling to \texttt{u}, \texttt{v}, and \texttt{sst} to improve stability for types with reduced dynamic range.
We disable this by setting all scaling values to $1$, as for purposes of stagnation detection, we do not want to mitigate issues caused by the relatively small dynamic range of low precision floating-point types like \texttt{Float16}.
However in our results we have not noticed a difference in the error of the results depending on whether scaling is or is not active, and the value $\eta$ that we are measuring stagnation and error of has no scaling directly applied to it.

\subsubsection{Stagnation and error in the time integration}

\begin{figure}[b!]
  \centering
  \begin{tikzpicture}
    \begin{axis}[
      title = {},
      xlabel = {Timestep $\mathrm{d}t$},
      ylabel = {Error (RMSE)},
      xmode = {log},
      ymode = {log},
    ]
    \addplot[color = blue, mark = o]
    table[
      x = {dt},
      y = {rmse},
      col sep = comma
    ]{swjl_fp64conv.dat};

    \end{axis}
  \end{tikzpicture}
  \caption{
    Convergence of error in ShallowWaters.jl in \texttt{Float64} arithmetic as timestep $\mathtt{dt}$ decreases.
    Root mean squared error is computed between the current simulation with timestep $\mathtt{dt}$ and the next simulation with timestep $\approx \mathtt{dt}/2$.
  }
  \label{fig:shallow64}
\end{figure}
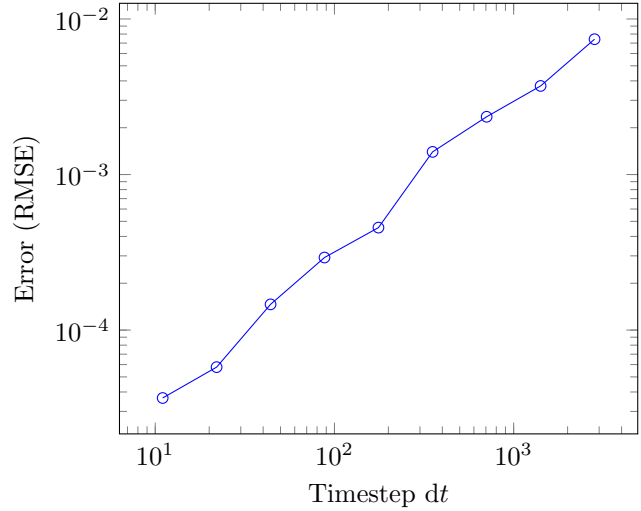

The default integration scheme used in \texttt{ShallowWaters.jl}, which we are measuring stagnation and error of, is a fourth-order Runge-Kutta method \cite{butcher16rk}.
At each timestep, four explicit forward integration steps are computed.
A linear combination of these evaluations is taken such that the result is an integration step with error $\mathcal{O}(\mathtt{dt}^4)$.
Some approaches add together the appropriate linear combination of RK integration steps first, then add this to the value of the integrated variable.
\texttt{ShallowWaters.jl} instead adds each value when computed, so each one step of integration with RK4 records four additions, all of which are measured with stagnation detection.

\begin{figure}[b!]
  \centering
  \begin{subfigure}[b]{0.475\textwidth}
    \centering
    \includegraphics[width=\textwidth]{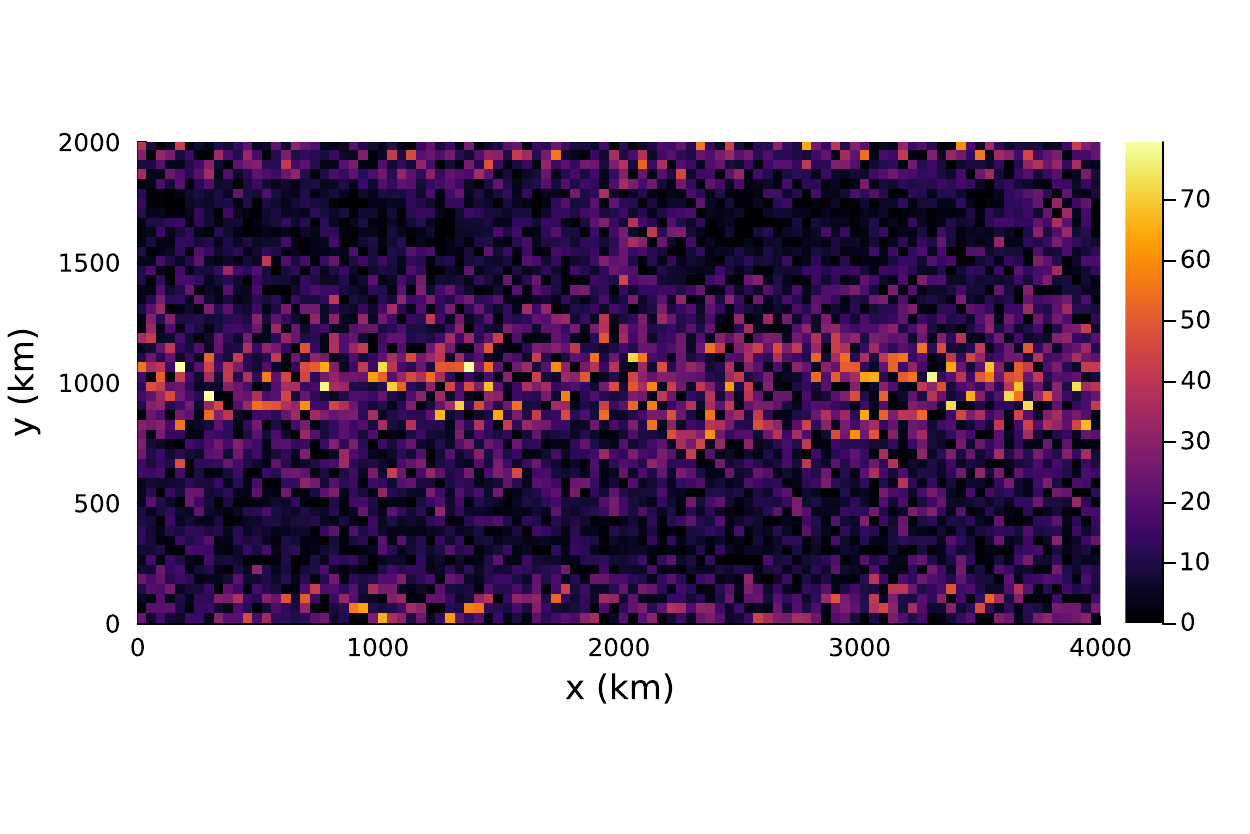}
  \end{subfigure}
  \hfill
  \begin{subfigure}[b]{0.475\textwidth}
    \centering
    \includegraphics[width=\textwidth]{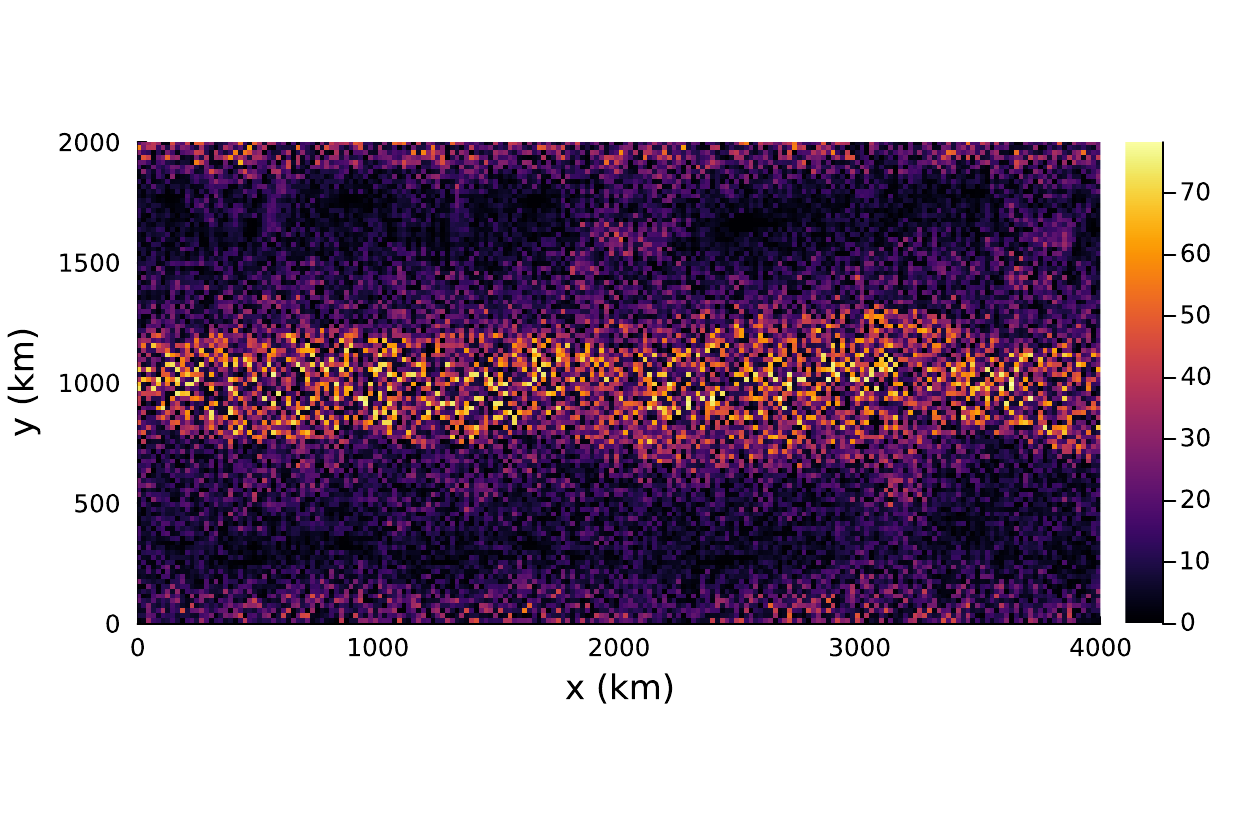}
  \end{subfigure}
  \caption{
    Stagnation of \texttt{ShallowWaters.jl} simulation measuring layer height $\eta$ using \texttt{Float16} arithmetic.
    Left column: default grid resolution $nx = 100$, right column: increased grid resolution $nx = 200$.
    Timestep $\mathtt{dt} = 176$ ($392728$ floating-point additions per entry) reduced to cause stagnation.
    Default timestep is $5656$ for $nx = 100$, $2828$ for $nx = 200$.
  }
  \label{fig:shallowstagmap}
\end{figure}

In \texttt{ShallowWaters.jl}, one of the optional parameters is the CFL number, which controls the ratio between the time and length discretisations.
By keeping $\mathtt{nx}$ fixed and adjusting the CFL number, we can control the timestep independently of other problem parameters.
The default configuration with $\mathtt{nx} = 100$ and $\mathtt{cfl} = 0.9$ integrating for $\texttt{Ndays} = 200$ days, which forms a default timestep $\mathtt{dt} = 5091$ seconds.
We test values of $\mathtt{cfl}$ from $1.0$ down to $2^{-9}$ which produces timestep sizes from $5656$ to $11$ seconds.
We verified that the result computed by \texttt{ShallowWaters.jl} is consistent in \texttt{Float64} arithmetic for different timestep values, and converges as the timestep size decreases, shown in Figure \ref{fig:shallow64}.
The most significant measurement of stagnation we observed is for $\mathtt{dt} = 176$, shown in Figure \ref{fig:shallowstagmap}.
Stagnation is signficant in many cases, where it exceeds $70\%$ of the summation in some parts of the domain.
At our smallest timestep measured $\mathtt{dt} = 11$, more than $6\times 10^6$ additions are performed for every entry in the solution matrix.

\begin{figure}[b!]
  \centering
  \begin{tikzpicture}
    \begin{axis}[
      axis y line = {right},
      title = {},
      xlabel = {Timestep $\mathrm{d}t$},
      ylabel = {Error (RMSE)},
      xmode = {log},
      ymode = {linear},
      enlarge y limits={0.1}
    ]
    \addplot[color = red, mark = x]
    table[
      x = {dt},
      y = {rmse},
      col sep = comma
    ]{swjl_stagerr.dat};

    \addplot[color = red, dashed, mark = x]
    table[
      x = {dt},
      y = {rmse_comp},
      col sep = comma
    ]{swjl_stagerr.dat};

    \end{axis}

    \begin{axis}[
      axis y line = {left},
      title = {},
      ylabel = {Mean stagnation (\%)},
      xmode = {log},
      ymode = {linear},
      enlarge y limits={0.1}
    ]
    \addplot[color = blue, mark = o]
    table[
      x = {dt},
      y = {stagnation},
      col sep = comma
    ]{swjl_stagerr.dat};
      
    \end{axis}
  \end{tikzpicture}
  \caption{
    Mean stagnation (blue) and RMSE (red) of \texttt{ShallowWaters.jl} simulations in \texttt{Float16} arithmetic with $\mathtt{nx} = 100$ grid resolution for different values of timestep $\mathtt{dt}$.
    Alternative RMSE (red, dashed) is obtained by performing the same computation in \texttt{Float16} arithmetic with compensated summation.
    Reference computation compared to every \texttt{Float16} result is run in entirely \texttt{Float64} with timestep $\mathtt{dt} = 11$.
  }
  \label{fig:shallowstagplot}
\end{figure}
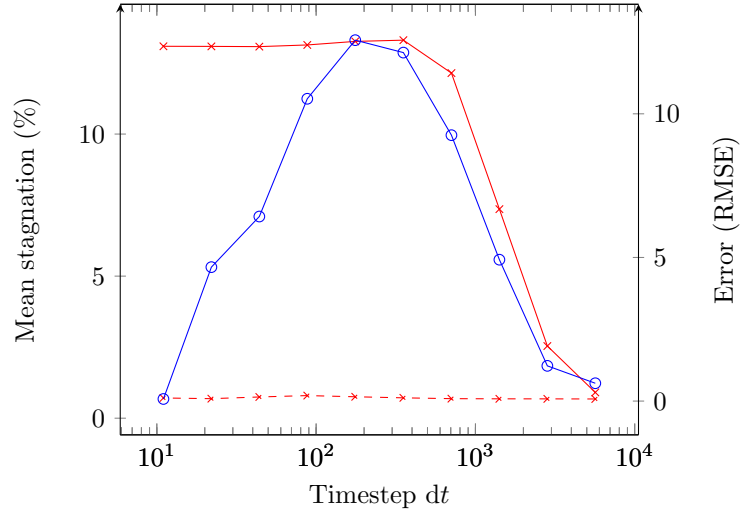

We show stagnation and error for \texttt{ShallowWaters.jl} together in Figure \ref{fig:shallowstagplot}.
The simulation is run normally in \texttt{Float16} arithmetic with stagnation detection to obtain measures of RMSE error and stagnation, and also with compensated summation \cite{high93s} to measure RMSE having applied a method to reduce stagnation.
For \texttt{Float16} arithmetic, RMSE and stagnation appear to grow at a similar rate as the timestep $\mathtt{dt}$ decreases.
The growth in RMSE with \texttt{Float16} does not occur with \texttt{Float64} arithmetic.
The value of mean stagnation is relatively low even at its maximum, but we attribute this to the structure of stagnation, which appears sparse when considering the whole domain (see Figure \ref{fig:shallowstagmap}).
The measure of RMSE reaches a plateau for values $\mathtt{dt} \leq 353$, and the most significant mean stagnation is for $\mathtt{dt} = 176$.
This growth would appear to show a strong relationship between stagnation and error.
However, for progressively smaller timesteps, we see the measure of stagnation drop down while the error remains the same, which may imply that error in \texttt{ShallowWaters.jl} with \texttt{Float16} arithmetic is coming from sources other than stagnation.
The measure of error is consistently large here, order $\sim10^1$, for small values of $\mathtt{dt}$.
For the simulation run with compensated summation \cite{high93s} to prevent stagnation, the RMSE error is greatly reduced.
Compensated summation is available as a feature in ShallowWaters.jl.
We do not measure stagnation with compensated summation enabled since the summation algorithm is more complex.

\begin{figure}[b!]
  \centering
  \begin{subfigure}[b]{0.3\textwidth}
    \includegraphics[width=\textwidth]{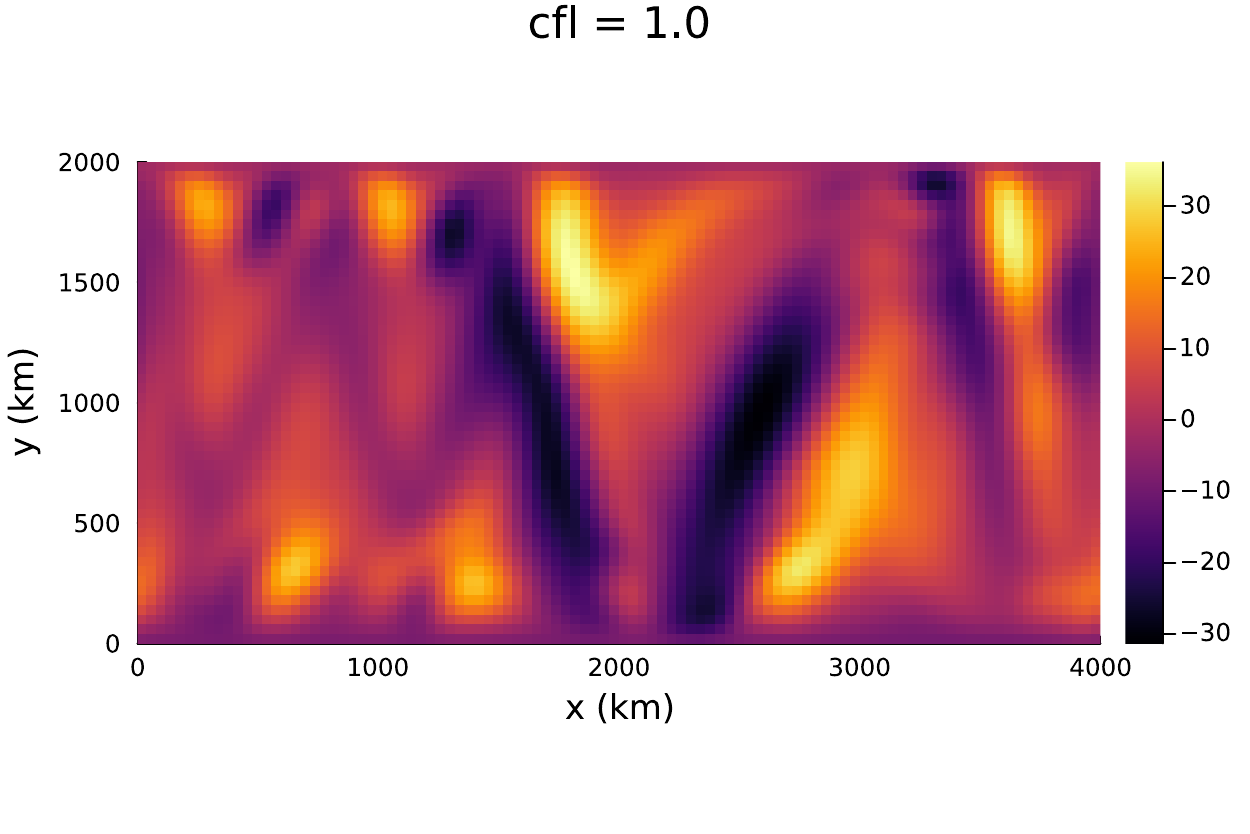}
  \end{subfigure}%
  \hfill%
  \begin{subfigure}[b]{0.3\textwidth}
    \includegraphics[width=\textwidth]{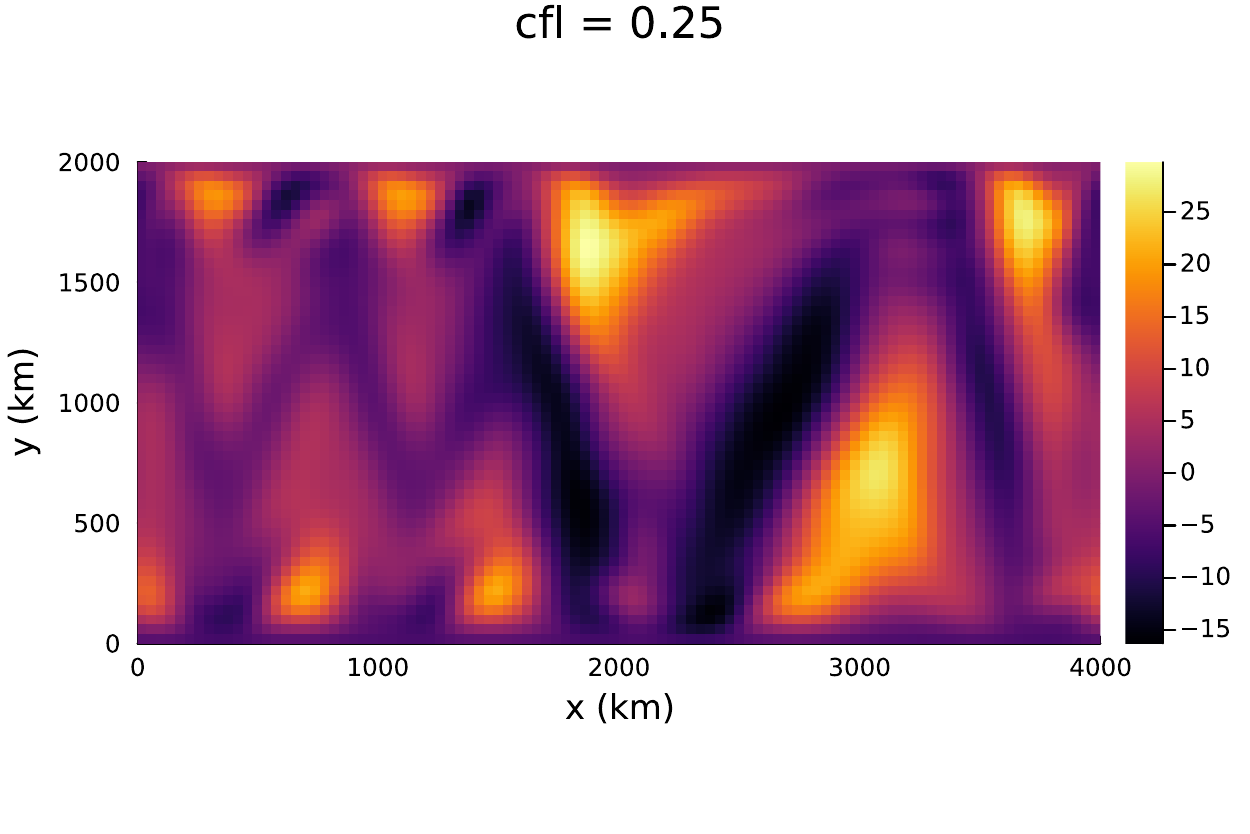}
  \end{subfigure}%
  \hfill%
  \begin{subfigure}[b]{0.3\textwidth}
    \includegraphics[width=\textwidth]{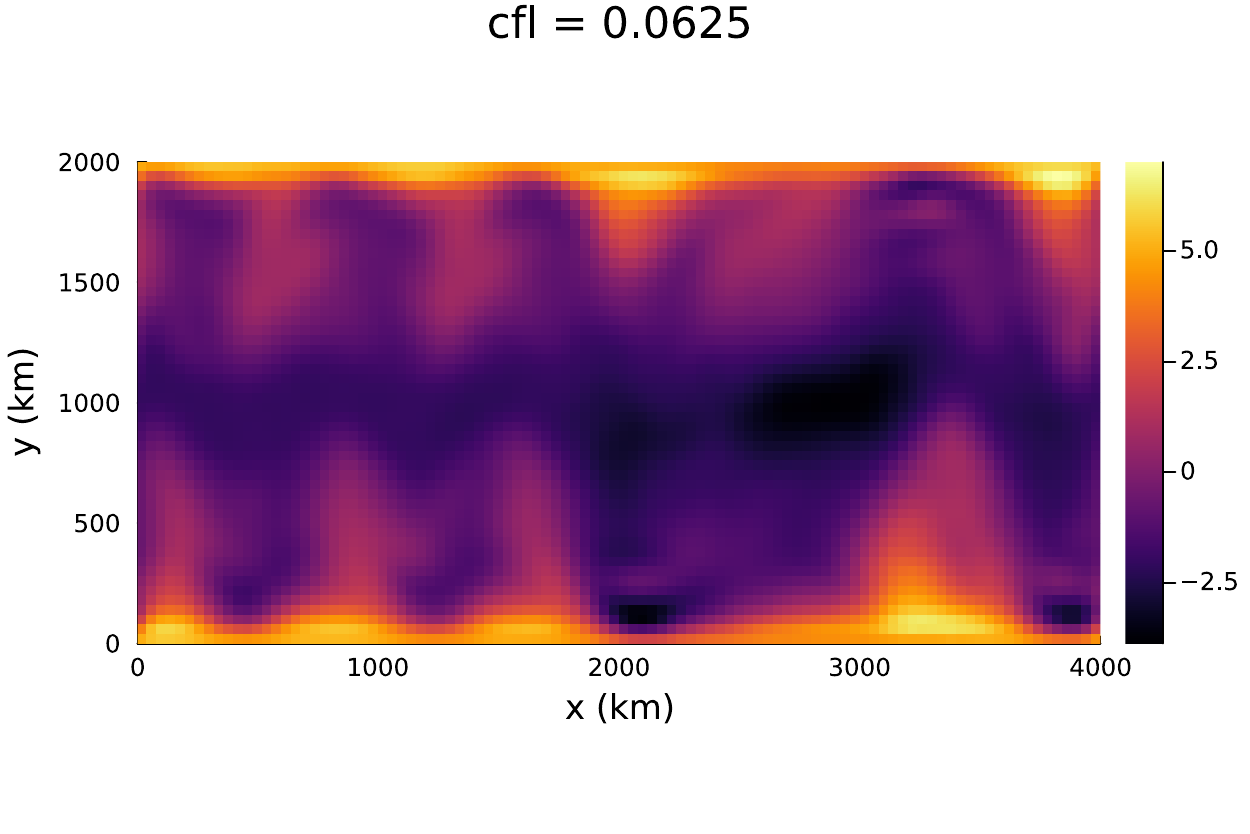}
  \end{subfigure}
  \caption{
    Shallow Waters surface layer height $\eta$ computed in \texttt{Float16} arithmetic for progressively smaller timestep sizes.
    Left: $\mathtt{dt} = 5656$. Middle: $\mathtt{dt} = 1414$. Right: $\mathtt{dt} = 353$.
    By RMSE, solution diverges from the \texttt{Float64} reference as the timestep size (controlled by $\mathtt{cfl}$) decreases. 
  }
  \label{fig:shallow16diverge}
\end{figure}

In \texttt{Float16} arithmetic the solution does not converge as the timestep is reduced.
Figure \ref{fig:shallow16diverge} shows plots of surface height $\eta$ for progressively decreasing $\mathtt{cfl}$. 
The divergence reflects the RMSE growth seen in Figure \ref{fig:shallowstagplot} as the timestep is decreased.
Running this simulation with compensated summation enabled prevents this divergence.

We see from our results that \texttt{ShallowWaters.jl} has significant stagnation for a given range of timesteps $\mathtt{dt}$.
In these cases, there is a sparse but significant horizontal region in the simulation where stagnation may affect more than $70\%$ of addition operations (Figure \ref{fig:shallowstagmap}).
The prevalence of stagnation could indicate to a user that there may be significant error in the computation.
We have compared this to RMSE as a single quantifier for the error of the computation.
The computation in \texttt{Float16} arithmetic measures RMSE error which is high for a signficant range of the small timesteps chosen, which could imply error is produced from sources other than stagnation.
In this result, high stagnation appears where error is large, but error is not always caused by stagnation.
Enabling compensated summation leads to a significant reduction in RMSE across the whole range of timestep values.
Compensated summation also mitigates stagnation.
Therefore \texttt{ShallowWaters.jl} is an example where preventing stagnation does lead to improved computational error, justifying stagnation detection as a tool.

However, it is worth mentioning that we reduce the timestep in order to cause stagnation in \texttt{Float16}, but stagnation instead decreases for smaller timesteps $\mathtt{dt} < 100$ and the solution accuracy does not improve.
For summations with extremely small timestep, the number of additions leads to a large accumulation of rounding error which dominates in \texttt{Float16}.
Reducing the size of the timestep to improve accuracy of the solution works in \texttt{Float64}, but in \texttt{Float16} requires a method to improve stability such as compensated summation.
Even then, we have shown there are configurations of \texttt{ShallowWaters.jl} where stagnation is relatively prominent, and enabling compensated summation here reduces the measure of error, justifying the use of stagnation detection.

\subsection{Benchmarking performance of stagnation detection}

All the benchmarks in this section were performed by replacing the Julia implementations \texttt{Base: sum}, \texttt{LinearAlgebra: dot} and \texttt{LinearAlgebra: *(::Matrix, ::Matrix)} with our own functions \texttt{lrsum}, \texttt{lrdot} and \texttt{lrmatmul}.
It can be observed from the results shown, both with and without stagnation detection, that the time complexity for each algorithm is as expected.
For each algorithm performed without stagnation detection, our own functions are much slower than their counterparts already implemented in Julia.

\subsubsection{Sums and vector inner products}

\begin{table}[b!]
  \centering
  \caption{
    Mean runtime of vector summation in double precision with and without stagnation detection on an Intel Ultra 7 165H CPU.
  }
  \begin{tabular}{llll}
  \toprule
  Summation length N & {Float64}               & {Float64, stag. detec.} & {Slowdown} \\                    
  \midrule
  1000000            & $409.500\mu \mathrm{s}$ & $5.719\mathrm{ms}$  & $14.0\times$ \\
  2000000            & $870.393\mu \mathrm{s}$ & $11.667\mathrm{ms}$ & $13.4\times$ \\
  4000000            & $1.654\mathrm{ms}$      & $24.471\mathrm{ms}$ & $14.8\times$ \\
  8000000            & $3.430\mathrm{ms}$      & $51.642\mathrm{ms}$ & $15.1\times$ \\                       
  \bottomrule
  \end{tabular}
  \label{tab:benchsum}
\end{table}

\begin{table}[b!]
  \centering
  \caption{
    Mean runtime of vector dot products in double precision with and without stagnation detection on an Intel Ultra 7 165H CPU.
  }
  \begin{tabular}{llll}
  \toprule
  {Summation length N} & {Float64}              & {Float64, stag. detec.}    & {Slowdown} \\
  \midrule
  1000000            & $821.265\mu \mathrm{s}$ & $7.752\mathrm{ms}$  & $9.4\times$ \\
  2000000            & $1.618\mathrm{ms}$      & $15.851\mathrm{ms}$ & $9.8\times$ \\
  4000000            & $3.274\mathrm{ms}$      & $32.141\mathrm{ms}$ & $9.8\times$ \\
  8000000            & $6.834\mathrm{ms}$      & $65.529\mathrm{ms}$ & $9.6\times$ \\               
  \bottomrule
  \end{tabular}
  \label{tab:benchdot}
\end{table}

\begin{table}[b!]
  \centering
  \caption{
    Performance benchmarking of matrix multiplication in double precision with and without stagnation detection on an Intel Ultra 7 165H CPU.
  }
  \begin{tabular}{llll}
  \toprule
  Dimensions                         & Duration  & Duration with stag. detec. & Slowdown     \\
  \midrule
  $(1000, 1000) \times (1000, 1000)$ & 778.868ms & 10.606s                            & $13.6\times$ \\
  $(1000, 2000) \times (2000, 1000)$ & 1.676s    & 21.117s                            & $12.6\times$ \\
  $(1000, 4000) \times (4000, 1000)$ & 3.607s    & 41.056s                            & $11.4\times$ \\
  $(2000, 1000) \times (1000, 2000)$ & 4.002s    & 42.322s                            & $10.6\times$ \\
  $(2000, 2000) \times (2000, 2000)$ & 8.106s    & 86.754s                            & $10.7\times$ \\
  $(2000, 4000) \times (4000, 2000)$ & 20.039s   & 179.167s                           & $8.9\times$  \\
  \bottomrule
  \end{tabular}
  \label{tab:benchmatmul}
\end{table}

Julia applies optimisations for mathematical functions on fundamental arithmetic types such as hardware floating point, which are not present for stagnation detection.
Therefore we have provided some arithmetic functions with these optimisations removed, such that the algorithms being implemented are the same regardless of if stagnation detection is being performed or not.
We have provided these functions for vector summation, dot product, matrix multiplication and classical and block LU factorisation both with partial pivoting.

For benchmarking purposes we generate random floating-point data by first using \texttt{rand()} to produce random unsigned integers of the appropriate size, then reinterpreting to floating point.
We use \texttt{BenchmarkTools: @benchmark} \cite{BenchmarkTools.jl-2016} in Julia for performing benchmarks, where the number of samples will decrease as the tests become more expensive, so the results especially for stagnation detection may not be noiseless.

Without stagnation detection, vector summation (Table \ref{tab:benchsum}) performs better than dot products (Table \ref{tab:benchdot}) but the performance for each is roughly the same when stagnation detection is active.
The functions we have written for these operations are extremely similar and modified from \texttt{LinearAlgebra: dot}, where the summation algorithm uses two-argument addition at each step while the dot product uses three-argument \texttt{muladd}.
Computing dot products can be seen to have a slowdown factor of approximately $9\times$ as seen in Table \ref{tab:benchdot}.
This increases significantly for matrix multiplication as shown in Table \ref{tab:benchmatmul}.

\subsubsection{Comparison to MPFR}

\begin{figure}[b!]
  \centering
  \begin{tikzpicture}
    \begin{axis}[
          title = {Floating-point vector sum performance},
          xlabel = {summation length},
          ylabel = {runtime (ms)},
          xmode = {log},
          ymode = {log},
          xtick = {1000000, 2000000, 4000000},
          ymax = {1500},
          enlarge y limits={0.1},
          legend pos = {north west},
          legend transposed = {true}
      ]

      \addplot[dotted, color = blue, thick]
      table[
          x = {n},
          y = {mp128},
          col sep = comma
      ]{benchsum.dat};
      \addlegendentry{MPFR ($p = 128$)}

      \addplot[dashed, color = blue, thick]
      table[
          x = {n},
          y = {stag},
          col sep = comma
      ]{benchsum.dat};
      \addlegendentry{fp64 stag. detec.}

      \addplot[solid, color = blue, thick]
      table[
          x = {n},
          y = {fp64},
          col sep = comma
      ]{benchsum.dat};
      \addlegendentry{fp64}

    \end{axis}
  \end{tikzpicture}
  \caption{
    Comparison of regular floating-point double precision recursive summation, summation with stagnation detection, and summation in $p=128$ precision using arbitrary precision with MPFR.
    The same unoptimised recursive summation algorithm was used in all computations.
  }
  \label{fig:mpfr}
\end{figure}
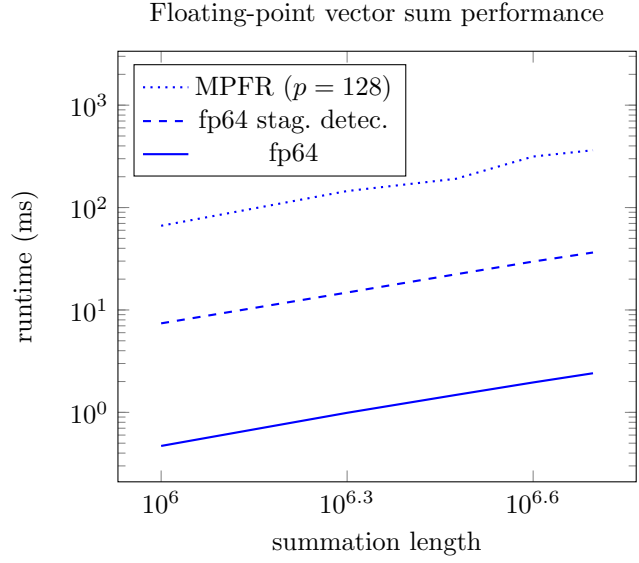

As a comparison of error measurement methods for floating point summation, we compare the computation of a vector sum with regular double precision arithmetic, double precision with stagnation detection, and $p=128$ arithmetic using MPFR arbitrary precision.
A plot is shown in Figure \ref{fig:mpfr}.
We can see that performing the sum with stagnation detection outperforms MPFR by a factor of approximately $10\times$, making it favourable for cases where stagnation can be a useful indicator of error.
It still provides significant overhead from regular floating point arithmetic, but provides a potentially favourable compromise between the performance and costs of the other two options shown.

\subsubsection{Storage cost}

\begin{table}[b!]
	\centering
	\caption{
		Sizes of Julia double precision scalars and arrays with and without stagnation detection.
	}
	\begin{tabular}{llll}
		\toprule
		Object                & Size      & Size with stagnation detection & Scale factor \\
		\midrule
		Scalar                & 8B        & 48B                            & $6\times$    \\
		Vector $N=100$        & 928B      & 5.71KiB                        & $6.3\times$  \\
		Vector $N=200$        & 1.62KiB   & 11.12KiB                       & $6.7\times$  \\
		Vector $N=400$        & 3.20KiB   & 22.06KiB                       & $6.9\times$  \\
		Vector $N = 1000000$  & 7.63MiB   & 53.41MiB                       & $7.0\times$  \\
		Matrix $(100, 100)$   & 78.21KiB  & 547.09KiB                      & $7.0\times$  \\
		Matrix $(200, 100)$   & 156.34KiB & 1.07MiB                        & $7.0\times$  \\
		Matrix $(400, 100)$   & 312.59KiB & 2.14MiB                        & $7.0\times$  \\
		Matrix $(1000, 1000)$ & 7.63MiB   & 53.41MiB                       & $7.0\times$  \\
		\bottomrule
	\end{tabular}
	\label{tab:sizes}
\end{table}

We have also shown the memory cost for stagnation detection in Table \ref{tab:sizes}.
All types belonging to \texttt{FloatSD} are structs with six fields, so we would expect a $6\times$ size increase for containers using \texttt{Float64SD} instead of \texttt{Float64} for example.
This is true for scalars but can be seen to grow for vectors and matrices.
It is worth noting that all the \texttt{FloatSD} types take up the same amount of space despite using different types for their \texttt{value} fields.
We believe this is because the Julia compiler allocates the system word size to every field in a struct.

\subsection{Alternative ways to measure stagnation}

\subsubsection{Consecutive sequences of absorptions}

\label{sssec:consecutive}

We considered several ways to define stagnation in a recursive summation before deciding on the current definition.
Previous to our work, the more commonly understood definition of stagnation applies to a recursive summation where every successive addition involves absorption.
We see no reason why the absorptions must be consecutive, thus we have relaxed the definition.
Here, we show examples measuring stagnation by recording the lengths of sequences of consecutive absorptions.
This is harder to implement in software, more computationally expensive, and does not provide more information than our actual approach.

When we state the \textit{longest stagnation}, this refers to a sequence of absorptions where all entries being absorbed are of the same sign.
If stagnation continues but the rounding direction is changed, this is considered to reset the sequence.
This is distinguished from our formal definition of stagnation, which itself is referred to as the \textit{number of absorptions}.

\begin{table}[b!]
	\centering
	\caption{
		Harmonic series computed in single precision, monitored with stagnation detection.
	}
	\begin{tabular}{llll}
		\toprule
		Number of terms & Final value of sum & Absorptions & Longest stagnation \\
		\midrule
    $1000000$ & $14.36$ & $0$ & $0$ \\
    $2000000$ & $15.31$ & $0$ & $0$ \\
    $3000000$ & $15.40$ & $902848$ & $902848$ \\
		$4000000$ & $15.40$ & $1902848$ & $1902848$ \\
		\bottomrule
	\end{tabular}
	\label{table:stagnation_harmonic}
\end{table}

When computing the harmonic series in finite precision, all the absorptions occur in a sequence at the end of the summation since the terms are decreasing.
Either implementation provides the same result for this problem.
See Table \ref{table:stagnation_harmonic} for the results.

\begin{table}[b!]
	\centering
	\caption{
		Summation of random single precision uniform $U[0, 1]$ values monitored with stagnation detection.
    Relative error is computed in reference to the sum in Julia's \texttt{Base.BigFloat} $256$-bit floating-point.
    For each $n$, data is produced randomly $100$ times and we take the (median) average of the results.
	}
	\begin{tabular}{llll}
		\toprule
		Number of terms & Absorptions & Longest stagnation & Sum rel. error \\
		\midrule
    $1000000$ & $10156$ & $3$ & $1.148\times 10^{-5}$ \\
    $2000000$ & $40672$ & $4$ & $1.398\times 10^{-5}$ \\
		$4000000$ & $162588$ & $5$ & $2.190\times 10^{-5}$ \\
		\bottomrule
	\end{tabular}
	\label{table:stagnation_uniform}
\end{table}

We also consider summation of random, uniformly-distributed values.
We measure stagnation in both ways we have considered, as well as the relative error of the summation.
A comparison is shown in Table \ref{table:stagnation_uniform}.
Observe that for $n = 10^6$, there are more than $10^5$ absorptions but the longest sequence of stagnation is length $3$.
If we were to only measure the longest sequence, we would be provided with much less information about the absorptions occurring.
Furthermore, the number of overall absorptions grows at a greater rate than the other metrics.


\begin{table}[b!]
	\centering
	\caption{
    Series expansion of $\mathrm{e}^{-2\pi}$ evaluated up to $64$ terms, evaluated in single precision in different orderings.
	}
	\begin{tabular}{lllll}
		\toprule
    Ordering & Absorptions down & Absorptions up & Longest stagnation & Sum rel. error \\
		\midrule
    default & $15$ & $15$ & $1$ & $2.548 \times 10^{-4}$ \\
    ascending & $0$ & $0$ & $0$ & $1.169 \times 10^{-3}$ \\
		descending & $15$ & $15$ & $1$ & $2.309\times 10^{-8}$ \\
		\bottomrule
	\end{tabular}
	\label{table:stagnation_exponential}
\end{table}

Higham~\cite{high93s} mentions one example of computing the series expansion of $\mathrm{e}^x$ about $0$ at $x = -2\pi$ by evaluating and summing the first $64$ terms in the series.
We measure the stagnation and the relative error of the summation, showing the results in Table \ref{table:stagnation_exponential}.
The result that descending ordering provides the most favourable error is replicated from the original work.
The data for stagnation, regardless of measurement, is the same for default and descending ordering despite the relative error being significantly different in both.
The error changes in descending ordering because the terms being absorbed are different.
Ascending ordering reports no stagnation, but the relative error is the greatest.
Stagnation is inappropriate for analysis of this problem, since the error culprit is catastrophic cancellation.

\subsubsection{Combination of addends}


Our stagnation detection tool records information from addition operations into a single accumulator.
If multiple variables of stagnation detection types are added together, their data will not be combined.
Any non-addition operation and any non-linear function returns a value with all the stagnation detection information reset.
We have considered alternative functionality, where information may be combined in addition, multiplication and division operations, and functions preserve stagnation data in general.
For vector sums and dot products, there is formally a single accumulator and these principles make no difference to the results.
However, for sums of sums and any extension of this form, such as multiple matrix multiplications, the contributions aggregate.
There is opportunity to argue which method is more rigorously justified in giving appropriate information on stagnation, as our implementation does not give a comprehensive history.
However it could also be said that by combining information from arbitrarily many operands, stagnation data can become cumbersome and meaningless to a user.
This is open to future analysis and discussion.


\section{Conclusion}


Our experiments with stagnation detection in Julia show that when performing recursive summation on single sign data
with a growing accumulator, there is a relationship between stagnation and relative error.
In these cases, we argue for the utility of stagnation detection as an alternative over other methods of dynamic floating point behaviour tracking and error detection.
We have provided a formal definition for stagnation and applied stagnation detection to a selection of example computations.
We have analysed the information provided and discussed the performance cost of stagnation detection.
Our software tool \texttt{StagnationDetection.jl} offers a simple way to implement stagnation detection in Julia programs, and we have provided extensive examples on usage.

Our definition of stagnation differs from how it is often informally understood, and we have briefly considered alternative definitions.
In future, it may be worth considering stagnation to apply strictly to consecutive operations, or to be defined on sequences of operations more complex than recursive summation.
Different approaches to defining and measuring stagnation could have different implications for the effectiveness of stagnation detection as a dynamic analysis tool and the complexity of its implementation.

Future work could explore stagnation detection applied to other problems we have not explored.
Iterative algorithms for solving linear systems, such as the conjugate gradient method, involve sequences of additions that could be monitored to detect stagnation.
Specifically for the conjugate gradient method, the most complex operation per step is matrix-vector multiplication, so implementing any manner of stagnation detection would be less expensive than it is for matrix-matrix multiplication, which we have provided benchmarks for.
Furthermore, while we have demonstrated results with \texttt{ShallowWaters.jl}, this is one example of a time integration and other problems solved with the same algorithm could exhibit more significant stagnation.
Further work towards improving the performance of stagnation detection, particularly for matrix arithmetic, may be worth considering.

\section*{Funding}

Mikaitis acknowledges funding by the Engineering and Physical Sciences Research Council (EPSRC) grant “Informing Future Numerical Standards by Determining Features of Non-Standard Mathematical Hardware”, ref. UKRI151.

\bibliography{references,njhigham,strings}

\end{document}